\documentclass[review]{elsarticle}

\usepackage{amssymb,amsmath,amsthm,bm,        %math
	makecell, booktabs, diagbox,              %table
	algorithm,algpseudocode,      			  %algorithm
	tikz,pgfplots,subfigure,adjustbox,        %plot
	natbib}                                   %bib
\usetikzlibrary{calc,matrix, shapes, arrows, positioning, fit,backgrounds}

\newtheorem{thm}{Theorem}
\newtheorem{cor}{Corollary}
\newtheorem{lem}{Lemma}

\numberwithin{equation}{section}
\usepackage{hyperref}

\journal{Advances in Applied Mathematics}

\begin{document}
\allowdisplaybreaks
\begin{frontmatter}

\title{A semi-bijective algorithm for saturated extended 2-regular simple stacks}

\author[mymainaddress]{Qianghui Guo\corref{mycorrespondingauthor}}
\cortext[mycorrespondingauthor]{Corresponding author}
\ead{guo@nankai.edu.cn}

\author[mymainaddress]{Yinglie Jin\corref{mycorrespondingauthor}}
\ead{yljin@nankai.edu.cn}

\author[mysecondaryaddress]{Lisa H. Sun\corref{mycorrespondingauthor}}
\ead{sunhui@nankai.edu.cn}

\author[mymainaddress,mythirdaddress]{Mingxing Weng}

\address[mymainaddress]{School of Mathematical Sciences, The Key Laboratory of Pure Mathematics and Combinatorics of Ministry of Education of China (LPMC), Nankai University, Tianjin 300071, P.R. China.}
\address[mysecondaryaddress]{Center for Combinatorics, The Key Laboratory of Pure Mathematics and Combinatorics of Ministry of Education of China (LPMC), Nankai University, Tianjin 300071, P.R. China.}
\address[mythirdaddress]{Department of Mathematics, Shanghai Jiao Tong University, Shanghai 200240,  P.R. China.}

\begin{abstract}
Combinatorics of biopolymer structures, especially enumeration of various RNA secondary structures and protein contact maps, is of significant interest for communities of both combinatorics and computational biology. 
However, most of the previous combinatorial enumeration results for these structures are presented in terms of generating functions, and few are explicit formulas. 
This paper is mainly concerned with finding explicit enumeration formulas for a particular class of biologically relevant structures, say,  saturated 2-regular simple stacks, whose configuration is related to protein folds in the 2D honeycomb lattice. We establish a semi-bijective algorithm that converts saturated 2-regular simple stacks into forests of small trees, which produces a uniform formula for saturated extended 2-regular simple stacks with any of the six primary component types. Summarizing the six different primary component types, we obtain a bivariate explicit formula for saturated extended 2-regular simple stacks with $n$ vertices and $k$ arcs. As consequences, the uniform formula can be reduced to Clote's results on $k$-saturated 2-regular simple stacks and the optimal 2-regular simple stacks, and Guo et al.'s result on the optimal extended 2-regular simple stacks.
\end{abstract}

\begin{keyword}
bijective algorithm\sep small tree\sep saturated stack\sep RNA secondary structure\sep protein contact map
\end{keyword}

\end{frontmatter}

\section{Introduction}

%diagram相关的重要研究工作
The diagram $G([n],E)$, a graph represented by drawing $n$ vertices in a horizontal line and arcs $(i,j)\in E$ in the upper halfplane, is a classical combinatorial structure closely related to set partitions and lattice paths \cite{chen-crossings-2005,chen-linked-2008,stein-class-1978I}. It attracts extensive studies by various motivations, one of which is from computational molecular biology, where the diagram is often used to model biopolymer structures like RNA secondary structures and protein contact maps.

Since Waterman set up a combinatorial framework for the study of RNA secondary structures in the 1970s \cite{waterman-secondary-1978,waterman-rna-1978}, combinatorial problems related to computational molecular biology, especially the combinatorial enumeration of various RNA secondary structures has attracted significant interest from both combinatorialists and theoretical biologists. 
For example, Waterman and his coworkers further obtained recurrence relations, explicit and asymptotic formulas for the number of several types of RNA secondary structures \cite{howell-computation-1980,schmitt-linear-1994,stein-some-1979,waterman-combinatorics-1979}.
Nebel and his coworkers provided enumerative results on statistical properties for (extended) RNA secondary structures using dot-bracket words and context-free {grammar methods} \cite{muller-combinatorics-2015,nebel-combinatorial-2002,nebel-quantitative-2009}.
Reidys et al. systematically studied RNA secondary structures with pseudoknots, and compiled their works in a monograph \cite{reidys-combinatorial-2010}. 
Clote and his coworkers proposed the concept of saturated RNA secondary structures and studied its enumeration problems  \cite{clote-combinatorics-2006,clote-asymptotics-2009,clote-asymptotic-2007,fusy-combinatorics-2014,waldispuhl-computing-2007}.

Recently, the combinatorial framework for protein contact map has also been initiated. When two amino acids in a protein fold come close enough to each other, they presumably form some kind of bond, which is called a contact.
The contact map of a protein fold is a graph that represents the patterns of contacts in the fold.
In combinatorics, the contact map is usually represented by arranging its amino acids on a horizontal line and drawing an arc between two residues if they form a contact. Contacts play a fundamental role in the study of protein structure and folding problems. 
Goldman et al. \cite{goldman-algorithmic-1999} showed that for any protein fold in 2D square lattice, the contact map can be decomposed into (at most) two stacks and one queue, which can be seen as generalizations of RNA secondary structures without and with pseudoknots, respectively. Istrail and Lam \cite{istrail-combinatorial-2009} proposed the question concerning generalizations of the Schmitt-Waterman counting formulas for RNA secondary structures \cite{schmitt-linear-1994} to enumerating stacks and queues, and they pointed out that this could provide insights into computing rigorous approximations of the partition function of protein folding in HP models. Thereafter, a series of works attacking the enumeration of stacks and queues were made by Chen, Guo and their co-authors \cite{chen-zigzag-2014,guo-combinatorics-2018,guo-enumeration-2016, guo-regular-2017}.

However, most of the above enumeration results are in the form of generating function, or generating function equation(s), or asymptotic formulas, few are explicit formulas.
This paper makes efforts to find general explicit enumeration formulas for a particular class of diagrams, say,  saturated extended 2-regular simple stacks, which emerges from the contact map of protein folds in the 2D honeycomb lattice.

It is known that in the classic hydrophobic-polar (HP) protein folding model \cite{dill-theory-1985}, the fold of a protein sequence is modeled as a self-avoiding walk in the 2D or 3D lattice. In different lattice models, the maximum vertex degree and {the minimum} arc length of the contact map can vary significantly. For instance, in a protein contact map in 2D square lattice, the degree of each internal vertex and terminal vertex is at most 2 and 3, respectively, and {the minimum arc length} is at least 3. While in the 2D honeycomb lattice, the degree of each internal vertex and terminal vertex is at most 1 and 2, respectively, and {the minimum arc length} is at least 5. Figure \ref{fig:honeycomb-lattice} \cite{guo-number-2022} shows the contact map of a protein fold in the 2D honeycomb lattice. For an investigation of various lattice models used for protein folding, we refer to \cite{pierri-lattices-2008}. 

\begin{figure}[H]
	\centering
	\begin{tikzpicture}
			\foreach \s in {0.45} {
				\draw[dash pattern=on 1pt off 1pt,thin] (0,2*\s) -- +(-150:\s) (2*\s*sin{60},5*\s) -- +(90:\s) (4*\s*sin{60},5*\s) -- +(90:\s) (6*\s*sin{60},2*\s) -- +(-30:\s);
				\draw[dash pattern=on 1pt off 1pt,thin] (0,-1*\s) -- ++(-30:\s) -- ++(30:\s) -- ++(-30:\s) -- ++(30:\s) -- ++(-30:\s) -- ++(30:\s);
				\draw[dash pattern=on 1pt off 1pt,thin] (-sin{60}*\s,cos{60}*\s) -- ++(90:\s) -- ++(150:\s) -- ++(90:\s) -- ++(30:\s) -- ++(90:\s) -- ++(30:\s) -- ++(90:\s) -- ++(30:\s) -- ++(-30:\s) -- ++(30:\s) -- ++(-30:\s) -- ++(30:\s) -- ++(-30:\s) -- ++(-90:\s) -- ++(-30:\s) -- ++(-90:\s) -- ++(-30:\s) -- ++(-90:\s) -- ++(-150:\s) -- ++(-90:\s);
				\foreach \i in {0,1,2,3}
				\foreach \j in {0,1} {\foreach \a in {30,150,-90} \draw[dash pattern=on 1pt off 1pt,thin] (2*sin{60}*\i*\s,3*\j*\s) -- +(\a:1*\s);}
				\foreach \i in {0,1,2}
				\foreach \j in {0,1} {\foreach \a in {30,150,-90} \draw[dash pattern=on 1pt off 1pt,thin] (2*\i*sin{60}*\s+sin{60}*\s,3*cos{60}*\s+3*\j*\s) -- +(\a:1*\s);}
				
				\draw[thick] (0,2*\s) -- ++(-30:\s) -- ++(-90:\s) -- ++(-30:\s) -- ++(-90:\s) -- ++(-30:\s) -- ++(30:\s) -- ++(90:\s) -- ++(150:\s) -- ++(90:\s) -- ++(150:\s)
				(0,2*\s) -- ++(90:\s) -- ++(30:\s) -- ++(-30:\s) -- ++(30:\s) -- ++(-30:\s) -- ++(-90:\s) -- ++(-30:\s) -- ++(30:\s) -- ++(90:\s) -- ++(150:\s) -- ++(90:\s) -- ++(150:\s) -- ++(-150:\s);
				
				\filldraw[thick] (0,2*\s) circle[radius=1pt] node[anchor=east]{\tiny 11} 
				++(-30:\s) circle[radius=1pt] node[anchor=east]{\tiny 10} 
				++(-90:\s) circle[radius=1pt] node[anchor=east]{\tiny 9} 
				++(-30:\s) circle[radius=1pt] node[anchor=east]{\tiny 8} 
				++(-90:\s) circle[radius=1pt] node[anchor=east]{\tiny 7} 
				++(-30:\s) circle[radius=1pt] node[anchor=south]{\tiny 6} 
				++(30:\s) circle[radius=1pt] node[anchor=west]{\tiny 5} 
				++(90:\s) circle[radius=1pt] node[anchor=west]{\tiny 4} 
				++(150:\s) circle[radius=1pt] node[anchor=west]{\tiny 3} 
				++(90:\s) circle[radius=1pt] node[anchor=west]{\tiny 2} 
				++(150:\s) circle[radius=1pt] node[anchor=west]{\tiny 1} 
				(0,2*\s) ++(90:\s) circle[radius=1pt] node[anchor=east]{\tiny 12} 
				++(30:\s) circle[radius=1pt] node[anchor=east]{\tiny 13} 
				++(-30:\s) circle[radius=1pt] node[anchor=south]{\tiny 14} 
				++(30:\s) circle[radius=1pt] node[anchor=south west]{\tiny 15} 
				++(-30:\s) circle[radius=1pt] node[anchor=south]{\tiny 16} 
				++(-90:\s) circle[radius=1pt] node[anchor=west]{\tiny 17} 
				++(-30:\s) circle[radius=1pt] node[anchor=west]{\tiny 18} 
				++(30:\s) circle[radius=1pt] node[anchor=west]{\tiny 19} 
				++(90:\s) circle[radius=1pt] node[anchor=west]{\tiny 20} 
				++(150:\s) circle[radius=1pt] node[anchor=west]{\tiny 21} 
				++(90:\s) circle[radius=1pt] node[anchor=west]{\tiny 22} 
				++(150:\s) circle[radius=1pt] node[anchor=west]{\tiny 23} 
				++(-150:\s) circle[radius=1pt] node[anchor=south]{\tiny 24};
			}			
	\end{tikzpicture}

	\begin{tikzpicture}
			\foreach \s in {0.45} 
			{\foreach \x in {1,...,24}
				\filldraw (\x *\s,0) circle[radius=1pt];
				\foreach \x in {1,...,24}
				\draw (\x*\s,-0.2) node{\footnotesize \x};
								
				\draw[dash pattern=on 1pt off 1pt]
				(1*\s,0) to [bend left=50] node[minimum size=5pt,inner sep=2pt,above, near end]{\footnotesize $S_1$} (10*\s,0)
				(3*\s,0) to [bend left=60] (8*\s,0);
				
				\draw[densely dash dot] 
				(1*\s,0) to [bend left=50] (14*\s,0)
				(2*\s,0) to [bend left=50] node[near start,above]{\footnotesize $Q$} (17*\s,0);
				
				\draw
				(15*\s,0) to [bend left=50] node[above]{\footnotesize $S_2$} (24*\s,0)
				(16*\s,0) to [bend left=50] (21*\s,0);}
	\end{tikzpicture}
	\caption{The contact map of a protein fold in the 2D honeycomb lattice. It can be decomposed into two stacks and one queue which are labelled by $S_1, S_2$, and $Q$, respectively.}\label{fig:honeycomb-lattice}
\end{figure}
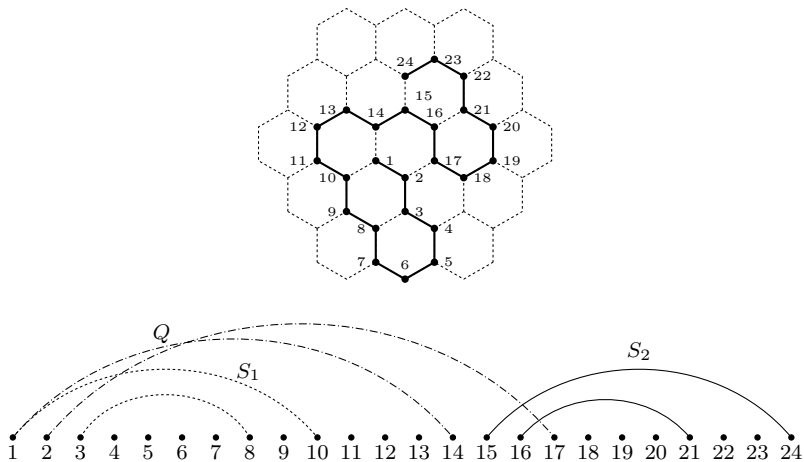

In a diagram, we say two arcs $(i,j)$ and $(k,l)$ form a nesting, if $i<k<l<j$, and a crossing if $i<k<j<l$. A noncrossing diagram is called a \textit{stack}, and a nonnesting diagram is called a \textit{queue}. Following \cite{chen-zigzag-2014, guo-regular-2017}, a structure (stack or queue) with arc length  at least $m$ is called \textit{$m$-regular}; a structure with the degree of each vertex bounded by one and two are called \textit{simple} and \textit{linear}, respectively. Actually, an RNA secondary structure can be viewed as a $2$-regular simple stack. Furthermore, 
{an \textit{extended $m$-regular simple stack} is an $m$-regular simple stack, except that the
two terminal vertices have a degree bounded by 2 instead of 1.}

The free energy minimization model plays an important role in the design of lots of structure prediction algorithms for RNA and protein \cite{waterman-rna-1978, zuker-comparison-1991}. In the classic Nussinov-Jacobson energy model \cite{nussinov-fast-1980}, the energy function is the negative of the number of base pairs (for RNA) or contacts (for protein) and the structures with minimum energy are called optimal. 
The number of optimal 2-regular simple stacks of length $n$, denoted by $LO_0(n)$, is given by Clote \cite[Corollary 13]{clote-combinatorics-2006} as follows.
\begin{equation}\label{equ:LO-0}
	LO_0(n)=\begin{cases}
		1,&\textup{if}\ n\ \textup{is odd},\\
		n(n+2)/8, &\textup{if}\ n\ \textup{is even},
	\end{cases}
\end{equation}
where $n\geq 1$, and $LO_0(0)=1.$
Guo et al. \cite[Theorem 2]{guo-number-2022} obtained the explicit expression for the number of the  optimal extended 2-regular simple stacks with $n$ vertices, denoted by $ELO_0(n)$, as follows.
\begin{equation}\label{equ:ELO-0}
	ELO_0(n)=\begin{cases}
		n-3, &\textup{if}\ n\ \textup{is even},\\
		(n^3-3n^2-7n+69)/12,&\textup{if}\ n\ \textup{is odd},
	\end{cases}
\end{equation}
where $n\geq 5$.

The {\textit{saturated}} structure, introduced by Zuker \cite{zuker-rna-1986}, is formally defined as the structure in which no arcs can be added without violating the constraints like arc length, vertex degree, and noncrossing. With respect to the Nussinov-Jacobson energy model, saturated secondary structures are actually local minima in the energy landscape. The combinatorial problem related to the number of saturated RNA secondary structures has been studied extensively \cite{clote-combinatorics-2006,clote-asymptotics-2009,clote-asymptotic-2007,fusy-combinatorics-2014,waldispuhl-computing-2007}. Following Zuker \cite{zuker-rna-1986}, Clote \cite{clote-combinatorics-2006} introduced the concept of $k$-saturated structure which is saturated and contains exactly $k$ fewer arcs than the optimal structures, and obtained recurrence relations for the number of $k$-saturated 2-regular simple stacks. Particularly, 0-saturated structures are just optimal structures.

In this paper, to find explicit formulas for saturated extended 2-regular simple stacks, we establish a semi-bijective algorithm that maps saturated extended 2-regular simple stacks to small forests. This algorithm is a composition of Schmitt and Waterman's bijection \cite{schmitt-linear-1994} between RNA secondary structures and linear trees, and a bijection between unlabelled linear trees and forests of small trees. The latter bijection can be seen as a variation of the bijection for Schr\"oder trees due to Chen \cite{chen-general-1990}. 
For saturated extended 2-regular simple stacks, we distinguish six types of primary components. 
By counting the resulting forests, we obtain a uniform formula for saturated extended 2-regular simple stacks with any of the six primary component types. 
As consequences, the uniform formula can be reduced to Clote's \cite{clote-combinatorics-2006} results on $k$-saturated 2-regular simple stacks and the optimal 2-regular simple stacks. 
By using this uniform formula for each primary component type, we obtain the main result of this paper, an explicit formula for the enumeration of saturated extended 2-regular simple stacks refined by the number of arcs.

\begin{thm}\label{thm:ELO(n,k)}
	Let $ELO(n,k)$ denote the total number of saturated extended 2-regular simple stacks with $n$ vertices and $k$ arcs. For any $n\geq 6,  k\geq 3$, we have
	\begin{equation}\label{equ:ELOnk}
		ELO(n,k)=\sum_{t=1}^{k+1}\sum_{i=1}^4\binom{k+1}{t}\binom{2t+k}{t-1}\binom{t}{n-2k-2-t+i}P_i(k,t),
	\end{equation}
	where 
	\begin{align*}
		P_1(k,t) = & \frac{2(k-t+1)}{(k+1)_3(k+2t)_2}\Big(k^{3}+(4 t-2) k^{2}+(2 t^{2}-4 t+1) k-2 t^{3}-4 t^{2}+4 t\Big),\\[5pt]
		P_2(k,t) = & \frac{k-t+1}{(k+1)_4(k+2t)_2}\Big(7 k^{4}+\left(22 t-22\right) k^{3}+\left(-3 t^{2}-49 t+17\right) k^{2}-(22 t^{3}\\
		& +13 t^{2}-37 t+2) k+6 t^{4}+44 t^{3}-4 t^{2}-18t\Big),\\[5pt]
		P_3(k,t) = & \frac{2(k-t+1)_2}{(k+1)_4(k+2t)_3}\Big(k^{4}+(7t-6)k^{3}+(15t^{2}-32t+13)k^{2}+(7t^{3}-44t^{2}\\
		& +45t-12)k-4t^{4}-18t^{3}+48t^{2}-30t+4\Big),\\[5pt]
		P_4(k,t) = & \frac{(t-1)(k+t+1)}{(k+1)_4(k+2t)_5}\Big(4k^6 + (16t - 42)k^5 + (8t^2 - 106t + 160)k^4 - (24t^3\\
		&- 26t^2  - 204t + 270)k^3 + (-10t^4 + 132t^3 - 244t^2 - 50t + 196)k^2\\
		& + (13t^5 - 8t^4- 217t^3+ 468t^2 - 208t - 48)k - 2t^6 - 20t^5 + 66t^4\\
		&  + 68t^3 - 304t^2 + 192t\Big),
	\end{align*}
	and $(n)_k=n(n-1)\cdots(n-k+1)$ denotes the $k$th falling factorial.
\end{thm}

It is worthwhile noting that Equation \eqref{equ:ELOnk} reduces to Guo et al.'s result \cite[Theorem 2]{guo-number-2022} for the number of optimal extended 2-regular simple stacks when taking $k=\lfloor \frac{n}{2}\rfloor$.

This paper is organized as follows. In Section \ref{sec:combinatorial-algorithm}, we present the semi-bijective algorithm as well as some basic definitions and notations. In Section \ref{sec:smrss}, we study the enumeration of saturated $m$-regular simple stacks. In Section \ref{sec:stack-given-arc-pattern}, we give a uniform explicit formula for enumerating saturated extended 2-regular simple stacks with six primary component types. At last, Section \ref{sec:saturated-extended-stacks} devotes to proving Theorem \ref{thm:ELO(n,k)} by using the uniform formula.

\section{The semi-bijective algorithm}\label{sec:combinatorial-algorithm}

In this section, we propose a semi-bijective algorithm that generates forests of small trees from 2-regular simple stacks. We first introduce some basic definitions and notations. 

A tree with a fixed root is called a \textit{rooted tree}. A \textit{linear tree} is a rooted tree together with a linear ordering on the set of children of each vertex in the tree. A linear tree of height one is called a \textit{small tree}, and a forest of small trees is called a \textit{small forest}, also known as \textit{meadow} in graph theory. 
In a linear tree, the \textit{fiber} of a vertex is the list of its children, a vertex with empty fiber is called a leaf, and all the other vertices are called internal. 
Obviously, the fiber of the root of a small tree is the list of its leaves. 
An internal vertex whose children are all leaves is called \textit{outmost internal}. A \textit{labelled tree} on $[n]$ is a tree in which the labels of all nodes is exactly $[n]$ with no repetition, where $[n]=\{1,2,\ldots,n\}$. A linear tree on $n$ vertices with each vertex labelled by a distinct number in $[n]$ is called a {\textit{labelled linear tree}}. See Figure \ref{fig:linear tree} for an example.
\begin{figure}[H]
	\centering
	\tikzstyle{vertex} = [circle, fill, minimum width=3pt, inner sep=0pt]
	\begin{tikzpicture}[grow=down,label distance=-2pt]
		\tikzstyle{level 1} = [level distance=8mm, sibling distance=15mm]
		\tikzstyle{level 2} = [level distance=8mm, sibling distance=7mm]
		\tikzstyle{level 3} = [level distance=8mm, sibling distance=7mm]
		\tikzstyle{level 4} = [level distance=8mm, sibling distance=5mm]
		\node[vertex,label=90:{\footnotesize 1}]{}
		child {
			node[vertex,label=180:{\footnotesize 2}]{}
			child {
				node[vertex,label=180:{\footnotesize 3}]{}
				child {
					node[vertex,label=180:{\footnotesize 8}]{}
					child {
						node[vertex,label=-90:{\footnotesize 12}]{}
						edge from parent
					}
					child {
						node[vertex,label=-90:{\footnotesize 4}]{}
						edge from parent
					}
					edge from parent
				}
				child {
					node[vertex,label=180:{\footnotesize 7}]{}
					child {
						node[vertex,label=-90:{\footnotesize 13}]{}
						edge from parent
					}
					edge from parent
				}
				child {
					node[vertex,label=-90:{\footnotesize 15}]{}
					edge from parent
				}
				edge from parent
			}
			child {
				node[vertex,label=-90:{\footnotesize 10}]{}
				edge from parent
			}
		}
		child {
			node[vertex,label=0:{\footnotesize 9}]{}
			child {
				node[vertex,label=-90:{\footnotesize 5}]{}
				edge from parent
			}
			child {
				node[vertex,label=-90:{\footnotesize 11}]{}
				edge from parent
			}
			child {
				node[vertex,label=0:{\footnotesize 16}]{}
				child {
					node[vertex,label=-90:{\footnotesize 6}]{}
					edge from parent
				}
				child {
					node[vertex,label=-90:{\footnotesize 14}]{}
					edge from parent
				}
				edge from parent
			}
			edge from parent
		};
	\end{tikzpicture}
	\caption{A labelled linear tree on $[16]$ with internal vertices $1,2,3,7,8,9,16$ and  outmost internal vertices $7,8,16$.}\label{fig:linear tree}
\end{figure}
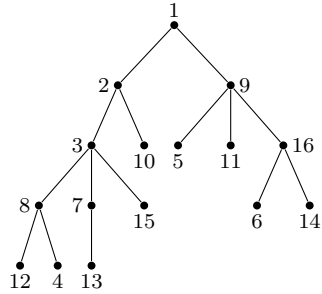

Let interval $\langle i,j\rangle$ denote the set $\{i+1,i+2,\ldots,j-1\}$, which may be empty. Let $[m,n]$ denote the set $\{m,m+1,\ldots,n\}$. Denote $\mathcal{R}(n,k)$ the set of 2-regular simple stacks on $[n]$ with $k$ arcs, and $\mathcal{T}(n,k)$  the set of unlabelled linear trees with $n$ vertices, in which $k$ vertices are internal. 

Recall the following bijection given by Schmitt and Waterman \cite{schmitt-linear-1994},
\begin{align}
    \varphi: \mathcal{R}(n,k)&\rightarrow \mathcal{T}(n-k+1,k+1),\label{eq:bijection phi}\\
    \notag S &\mapsto T,
\end{align}
which is defined as follows.  

Denote the set of isolated vertices of $S$ by $I$. Let $V$ be the set $\{[i,j]:(i,j)\in S\}\cup \{[0,n+1]\}\cup I$. Partially order $V$ by set inclusion and then the Hasse diagram of $V$ is a rooted tree having $n-k+1$ vertices in which $k+1$ vertices are internal. The linear order of the set $I$ of terminal vertices gives this tree a linear structure. By removing all the labels of this linear tree, we obtain $T\in \mathcal{T}(n-k+1,k+1)$.

Figure \ref{fig:saturated RNA and linear tree} illustrates the bijection $\varphi$.  
\begin{figure}[H]
	\centering
	\subfigure[A saturated 2-regular simple stack $S$]{\adjustbox{valign=c}{\begin{tikzpicture}
				\foreach \s in {.4} {
				\clip (0.9*\s,-1.5*\s) rectangle (21.3*\s,3.5*\s);
				\foreach \x in {1,...,21}
				\filldraw (\x *\s,0) circle[radius=1pt];
				\foreach \x in {1,...,21}
				\draw (\x*\s,-0.2) node{\footnotesize \x};
				\foreach \x/\y in {1/13,2/11,3/6,7/9,14/21,17/20}
				\draw (\x*\s,0) to [bend left=60] (\y*\s,0);}
	\end{tikzpicture}}}
	
	\subfigure[Hasse diagram of $V$]{\adjustbox{valign=c}{\begin{tikzpicture}[grow=down,label distance=-2pt]
				\clip (-2.6,-3.8) rectangle (2.4,0.4);
				\tikzstyle{vertex} = [circle, fill, minimum width=3pt, inner sep=0pt]
				\tikzstyle{level 1} = [level distance=8mm, sibling distance=15mm]
				\tikzstyle{level 2} = [level distance=8mm, sibling distance=7mm]
				\tikzstyle{level 3} = [level distance=8mm, sibling distance=8mm]
				\tikzstyle{level 4} = [level distance=8mm, sibling distance=5mm]
				\node[vertex,label=90:{\footnotesize [0,22]}]{}
				child {
					node[vertex,label=180:{\footnotesize [1,13]}]{}
					child {
						node[vertex,label=180:{\footnotesize [2,11]}]{}
						child {
							node[vertex,label=180:{\footnotesize [3,6]}]{}
							child {
								node[vertex,label=-90:{\footnotesize 4}]{}
								edge from parent
							}
							child {
								node[vertex,label=-90:{\footnotesize 5}]{}
								edge from parent
							}
							edge from parent
						}
						child {
							node[vertex,label=180:{\footnotesize [7,9]}]{}
							child {
								node[vertex,label=-90:{\footnotesize 8}]{}
								edge from parent
							}
							edge from parent
						}
						child {
							node[vertex,label=-90:{\footnotesize 10}]{}
							edge from parent
						}
						edge from parent
					}
					child {
						node[vertex,label=-90:{\footnotesize 12}]{}
						edge from parent
					}
				}
				child {
					node[vertex,label=0:{\footnotesize [14,21]}]{}
					child {
						node[vertex,label=-90:{\footnotesize 15}]{}
						edge from parent
					}
					child {
						node[vertex,label=-90:{\footnotesize 16}]{}
						edge from parent
					}
					child {
						node[vertex,label=0:{\footnotesize [17,20]}]{}
						child {
							node[vertex,label=-90:{\footnotesize 18}]{}
							edge from parent
						}
						child {
							node[vertex,label=-90:{\footnotesize 19}]{}
							edge from parent
						}
						edge from parent
					}
					edge from parent
				};
	\end{tikzpicture}}}
	\hspace{.5cm}
	\subfigure[Unlabelled linear tree $\varphi(S)$]{\adjustbox{valign=c}{\begin{tikzpicture}[grow=down]
	            \clip (-2.2,-3.8) rectangle (2,0.4);
				\tikzstyle{vertex} = [circle, fill, minimum width=3pt, inner sep=0pt]
				\tikzstyle{level 1} = [level distance=8mm, sibling distance=15mm]
				\tikzstyle{level 2} = [level distance=8mm, sibling distance=7mm]
				\tikzstyle{level 3} = [level distance=8mm, sibling distance=7mm]
				\tikzstyle{level 4} = [level distance=8mm, sibling distance=5mm]
				\node[vertex]{}
				child {
					node[vertex]{}
					child {
						node[vertex]{}
						child {
							node[vertex]{}
							child {
								node[vertex]{}
								edge from parent
							}
							child {
								node[vertex]{}
								edge from parent
							}
							edge from parent
						}
						child {
							node[vertex]{}
							child {
								node[vertex]{}
								edge from parent
							}
							edge from parent
						}
						child {
							node[vertex]{}
							edge from parent
						}
						edge from parent
					}
					child {
						node[vertex]{}
						edge from parent
					}
				}
				child {
					node[vertex]{}
					child {
						node[vertex]{}
						edge from parent
					}
					child {
						node[vertex]{}
						edge from parent
					}
					child {
						node[vertex]{}
						child {
							node[vertex]{}
							edge from parent
						}
						child {
							node[vertex]{}
							edge from parent
						}
						edge from parent
					}
					edge from parent
				};
	\end{tikzpicture}}}
	\caption{A saturated 2-regular simple stack in $\mathcal{R}(21,6)$ and its corresponding unlabelled linear tree in $\mathcal{T}(16,7)$.} \label{fig:saturated RNA and linear tree}
\end{figure}
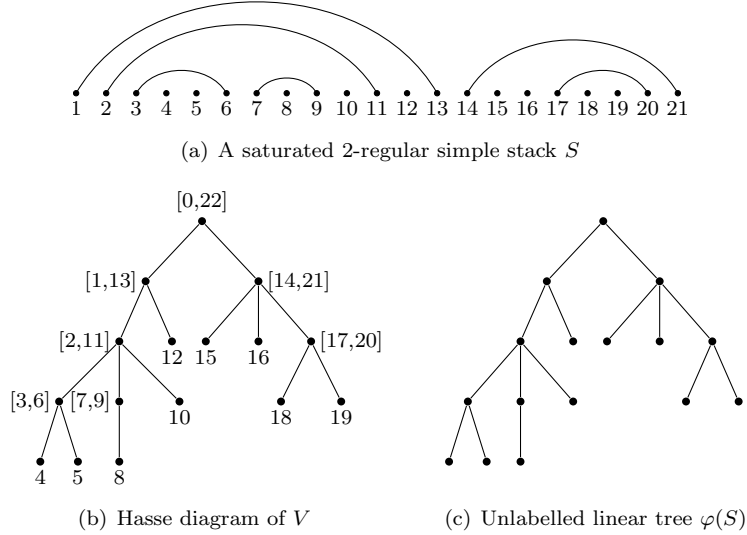

Based on the bijection $\varphi$ and the enumeration results on unlabelled linear trees \cite{chen-general-1990}, Schmitt and Waterman \cite{schmitt-linear-1994} derived the number of 2-regular simple stacks on $[n]$ with $k$ arcs
\begin{equation}
    |\mathcal{R}(n,k)|=\frac{1}{k}\binom{n-k}{k+1}\binom{n-k-1}{k-1},\qquad n, k\geq 1.
\end{equation}

Next, we give a bijective algorithm that constructs small forests from labelled linear trees. The idea of our algorithm originates from Chen's decomposition algorithm for Schr\"oder trees \cite{chen-general-1990}, which has been used to obtain many classical results for enumerations of trees.

Denote $\mathcal{LT}(n,k)$ the set of labelled linear trees on $n$ vertices with $k$ internal vertices. Let $\mathcal{F}(n,k)$ be the set of forests on $[n]$ with $k$ small trees such that all the roots are assigned labels less than or equal to $n-k+1$, and all the vertices with labels greater than $n-k+1$ are asterisked. {In other words, the last $k-1$ labels of $[n]$ are asterisked and cannot be roots.} Obviously, roots must be unasterisked. {Note that the set of labels of the nodes in $\mathcal{LT}(n,k)$ is exactly $[n]$, and that the set of all labels of the small trees in $\mathcal{F}(n,k)$ is also exactly $[n]$ with no repetition in the labels among the trees in the forest.}

We define the bijection between $\mathcal{LT}(n,k)$ and {$\mathcal{F}(n+k-1,k)$} 
\begin{align}
    \psi: \mathcal{LT}(n,k) &\rightarrow \mathcal{F}(n+k-1,k)\\
    \notag LT &\mapsto F,
\end{align}
as follows: for a given $LT\in\mathcal{LT}(n,k)$, 
\begin{enumerate} 
	\item Initialize $F=\varnothing$;
	\item Suppose the label of the largest outmost internal vertex of $LT$ is $i$, and denote its fiber by $B$. Then add a small tree with root $i$ and fiber $B$ into $F$;
	\item In $LT$, remove the fiber $B$ and relabel the vertex $i$ by $(n+1)^*$;
	\item Repeat step 2 and step 3, and relabel the largest outmost internal vertex $i$ in each time by $(n+2)^*, \ldots, (n+k-1)^*$ subsequently, until all the internal vertices have been asterisked except the root. Then add the small tree   with the root  of $LT$ into $F$.
\end{enumerate}
Given $F\in{\mathcal{F}(n+k-1,k)}$, the inverse map can be done as follows.
\begin{enumerate} 
	\item In $F$, among the trees with no asterisked vertex, select the one whose root label is the largest, denote that tree by $T$;
	\item Find the tree $T^*$ in $F$ that contains $(n+1)^*$, then update $F$ by replacing $(n+1)^*$ with $T$ in $T^*$;
	\item Repeat step 1 and step 2 for vertices $(n+2)^*,\ldots, (n+k-1)^*$ until there is only one tree in $F$;
	\item Let $LT$ be the only tree in $F$.
\end{enumerate}

It is straightforward to see that the above two maps are inverse to each other, and thus $\psi$ is a bijection. Figure \ref{fig:meadow-decompose} shows an example.

\begin{figure}[H]
	\centering
	\tikzstyle{vertex} = [circle, fill, minimum width=3pt, inner sep=0pt]
	\begin{minipage}{.45\linewidth}
		\centering
		\subfigure{\adjustbox{valign=c}{
				\begin{tikzpicture}[grow=down]
					\tikzstyle{level 1} = [level distance=8mm, sibling distance=11mm]
					\tikzstyle{level 2} = [level distance=8mm, sibling distance=5mm]
					\tikzstyle{level 3} = [level distance=8mm, sibling distance=6mm]
					\tikzstyle{level 4} = [level distance=8mm, sibling distance=5mm]
					\node[vertex,label=90:{\footnotesize 1}]{}
					child {
						node[vertex,label=180:{\footnotesize 2}]{}
						child {
							node[vertex,label=180:{\footnotesize 3}]{}
							child {
								node[vertex,label=180:{\footnotesize 8}]{}
								child {
									node[vertex,label=-90:{\footnotesize 12}]{}
									edge from parent
								}
								child {
									node[vertex,label=-90:{\footnotesize 4}]{}
									edge from parent
								}
								edge from parent
							}
							child {
								node[vertex,label=180:{\footnotesize 7}]{}
								child {
									node[vertex,label=-90:{\footnotesize 13}]{}
									edge from parent
								}
								edge from parent
							}
							child {
								node[vertex,label=-90:{\footnotesize 15}]{}
								edge from parent
							}
							edge from parent
						}
						child {
							node[vertex,label=-90:{\footnotesize 10}]{}
							edge from parent
						}
					}
					child {
						node[vertex,label=0:{\footnotesize 9}]{}
						child {
							node[vertex,label=-90:{\footnotesize 5}]{}
							edge from parent
						}
						child {
							node[vertex,label=-90:{\footnotesize 11}]{}
							edge from parent
						}
						child {
							node[vertex,label=0:{\footnotesize 16}]{}
							child {
								node[vertex,label=-90:{\footnotesize 6}]{}
								edge from parent
							}
							child {
								node[vertex,label=-90:{\footnotesize 14}]{}
								edge from parent
							}
							edge from parent
						}
						edge from parent
					};
		\end{tikzpicture}}}
		\subfigure{\adjustbox{valign=c}{
				\begin{tikzpicture}
					\foreach \s in {1} {
					\clip (-.2*\s,-.7*\s) rectangle (\s,.5*\s);
					\draw[arrows={-latex}] (0,0)--(\s,0) ;
					\draw[arrows={-latex}] (\s,-.2*\s)--(0,-.2*\s);
					\draw 	(0.5*\s,0)	node[above]{$\psi$} 
					(0.5*\s,-0.2*\s)	node[below]{$\psi^{-1}$};}
		\end{tikzpicture}}}
	\end{minipage}
	\begin{minipage}{.45\linewidth}
		\centering
		\begin{tikzpicture}
			\tikzstyle{level 1} = [level distance=8mm, sibling distance=7mm]
			\node[vertex,label=90:{\footnotesize 1}]{}
			child {
				node[vertex,label=-90:{\footnotesize 22*}]{}
				edge from parent
			}
			child {
				node[vertex,label=-90:{\footnotesize 18*}]{}
				edge from parent
			};
		\end{tikzpicture}
		\begin{tikzpicture}[grow=down]
			\tikzstyle{level 1} = [level distance=8mm, sibling distance=7mm]
			\node[vertex,label=90:{\footnotesize 2}]{}
			child {
				node[vertex,label=-90:{\footnotesize 21*}]{}
				edge from parent
			}
			child {
				node[vertex,label=-90:{\footnotesize 10}]{}
				edge from parent
			};
		\end{tikzpicture}
		\begin{tikzpicture}
			\tikzstyle{level 1} = [level distance=8mm, sibling distance=7mm]
			\node[vertex,label=90:{\footnotesize 3}]{}
			child {
				node[vertex,label=-90:{\footnotesize 19*}]{}
				edge from parent
			}
			child {
				node[vertex,label=-90:{\footnotesize 20*}]{}
				edge from parent
			}
			child {
				node[vertex,label=-90:{\footnotesize 15}]{}
				edge from parent
			};
		\end{tikzpicture}
		
		\begin{tikzpicture}[grow=down]
			\tikzstyle{level 1} = [level distance=8mm, sibling distance=7mm]
			\node[vertex,label=90:{\footnotesize 7}]{}
			child {
				node[vertex,label=-90:{\footnotesize 13}]{}
				edge from parent
			};
		\end{tikzpicture}
		\begin{tikzpicture}[grow=down]
			\tikzstyle{level 1} = [level distance=8mm, sibling distance=7mm]
			\node[vertex,label=90:{\footnotesize 8}]{}
			child {
				node[vertex,label=-90:{\footnotesize 12}]{}
				edge from parent
			}
			child {
				node[vertex,label=-90:{\footnotesize 4}]{}
				edge from parent
			};
		\end{tikzpicture}
		\begin{tikzpicture}[grow=down]
			\tikzstyle{level 1} = [level distance=8mm, sibling distance=7mm]
			\node[vertex,label=90:{\footnotesize 9}]{}
			child {
				node[vertex,label=-90:{\footnotesize 5}]{}
				edge from parent
			}
			child {
				node[vertex,label=-90:{\footnotesize 11}]{}
				edge from parent
			}
			child {
				node[vertex,label=-90:{\footnotesize 17*}]{}
				edge from parent
			};	
		\end{tikzpicture}
		\begin{tikzpicture}[grow=down]
			\tikzstyle{level 1} = [level distance=8mm, sibling distance=7mm]
			\node[vertex,label=90:{\footnotesize 16}]{}
			child {
				node[vertex,label=-90:{\footnotesize 6}]{}
				edge from parent
			}
			child {
				node[vertex,label=-90:{\footnotesize 14}]{}
				edge from parent
			};
		\end{tikzpicture}
	\end{minipage}
	
	\caption{A labelled linear tree in $\mathcal{LT}(16, 7)$ and its corresponding small forest in {$\mathcal{F}(22, 7)$}.} \label{fig:meadow-decompose}
\end{figure}
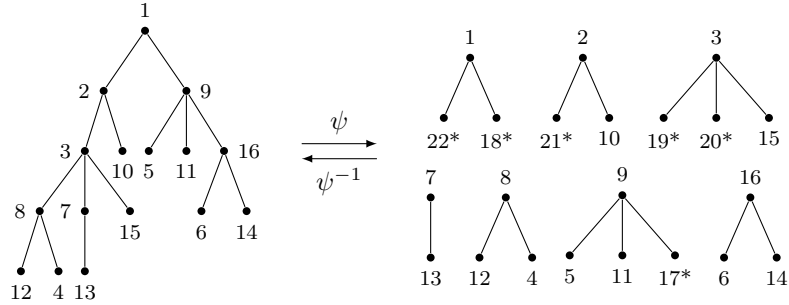

Through the construction of $\varphi$ and $\psi$, it is interesting to observe the  correspondences between properties of 2-regular simple stacks, linear trees, and small forests as shown in Table \ref{tab:relation}.

\begin{table}[H]
	\centering
	\caption{Correspondences between 2-regular simple stacks, linear trees and small forest.}\label{tab:relation}
	\resizebox{\textwidth}{!}{\begin{tabular}{m{0.33\textwidth}<{\centering}|m{0.33\textwidth}<{\centering}|m{0.33\textwidth}<{\centering}}
			\Xhline{0.3mm}
			$\mathcal{R}(n,k)$      & $\mathcal{T}(n-k+1,k+1)$            	  & $\mathcal{F}(n+1,k+1)$ \\\Xhline{0.3mm}
			$k$ arcs                 & $k+1$ internal vertices                    & $k+1$ small trees \\\hline
			$n-2k$ isolated vertices & $n-2k$ leaves                          	  & $n-2k$ unasterisked leaves   \\\hline
%			$p$ hairpins             & $p$ outmost internal vertices              & $p$ small trees with no asterisked vertices \\\hline
			arc of length $m$ and covering $m-1$ isolated vertices
			& outmost internal vertice with degree $m-1$ & small trees with $m-1$ unasterisked leaves \\\hline
			$b$ visible vertices     & $b$ leaves in the children of the root	  &	$b$ unasterisked leaves in the small tree that contains $(n+1)^*$	 \\\Xhline{0.3mm}
	\end{tabular}}
\end{table}

Remark that a vertex in a stack is called \textit{visible} if it is not covered by any arc.
For the last cell in Table \ref{tab:relation}, note that if $S\in \mathcal{R}(n,k)$ has only isolated vertices, then $T=\varphi(S)$ is a small tree, all the vertices of $S$ are visible, and $T$ remains unchanged under bijection $\psi$, so all the leaves of $\psi(T)$ are unasterisked.

Now we are ready to give our semi-bijective algorithm that maps a 2-regular simple stack to its corresponding small forest. The main idea is that first convert the 2-regular simple stack $S$ to an unlabeled linear tree, then label the vertices of the tree $\varphi(S)$, and finally apply bijection $\psi$ to produce the desired small forests. The full algorithm is stated as follows.

\begin{algorithm}[H]
    \renewcommand{\thealgorithm}{}
	\renewcommand{\algorithmicrequire}{\textbf{Input:}}
	\renewcommand{\algorithmicensure}{\textbf{Output:}}
	\caption{The semi-bijective algorithm {\tt STF}}\label{algorithm0}
	\begin{algorithmic}
		\Require  A 2-regular simple stack $S\in\mathcal{R}(n,k)$.
		\Ensure A set of small forests $\mathcal{F}\subseteq\mathcal{F}(n+1,k+1)$.
		\State \textbf{Step 1.} Set $T=\varphi(S)\in\mathcal{T}(n-k+1,k+1)$.
		\State \textbf{Step 2.} Label the vertices of $T$ distinctly with numbers in $[n-k+1]$ arbitrarily, and let $\mathcal{LT}$ be the set of the $(n-k+1)!$ labelled linear trees. 
		\State \textbf{Step 3.} Set $\mathcal{F}=\{F \mid F=\psi(T), T \in \mathcal{LT}\}$.	
	\end{algorithmic}
\end{algorithm}

In the following, we will use the algorithm {\tt STF} to enumerate saturated $m$-regular simple stacks, and a variant of {\tt STF} will be used to enumerate saturated extended 2-regular simple stacks.

\section{Saturated $m$-regular simple stacks}\label{sec:smrss}

In this section, we are concerned with the enumeration of saturated $m$-regular simple stacks through the semi-bijective algorithm {\tt STF}. 

Denote the set of saturated $m$-regular simple stacks on $[n]$ with $k$ arcs by $\mathcal{R}_s(n,k;m)$, and denote $ \mathcal{F}_s(n+1,k+1;m)$ the set of small forests generated from  $\mathcal{R}_s(n,k;m)$ by {\tt STF}, that is
\[
   \mathcal{F}_s(n+1,k+1;m)=\bigcup_{S\in\mathcal{R}_s(n,k;m)}\text{\tt STF}(S).
\]
Let ${R}_s(n,k;m)=|\mathcal{R}_s(n,k;m)|$ and ${F}_s(n,k;m)=|\mathcal{F}_s(n,k;m)|$. The following lemma characterizes the set $\mathcal{F}_s(n+1,k+1;m)$.

\begin{lem}\label{lem:R-s2F-s}
   {Assume $F$ is a small forest, then $F\in\mathcal{F}_s(n+1,k+1;m)$ if and only if $F\in\mathcal{F}(n+1,k+1)$ and the fiber of any small tree in $F$ satisfies the following properties:}
    \begin{enumerate}[(P1)]
        \item When the fiber contains no asterisked vertex, it must be of $m-1$ or $m$ unasterisked vertices.
        \item When the fiber contains asterisked vertices, {then the unasterisked vertices are all
        consecutive, and there are no more than $m$ of them.} Here unasterisked vertices are \textit{consecutive} if no asterisked vertex appears between unasterisked ones.
    \end{enumerate}
    Moreover, we have
    \begin{equation}\label{equ:R-s2F-s}
        {R}_s(n,k;m)=\frac{{F}_s(n+1,k+1;m)}{(n-k+1)!}.
    \end{equation}
\end{lem}

\begin{proof}
Obviously, $\mathcal{F}_s(n+1,k+1;m)\subseteq\mathcal{F}(n+1,k+1)$. Through the bijection $\varphi$ defined by \eqref{eq:bijection phi}, one can see that the saturated $m$-regular simple stacks correspond to those unlabelled linear trees satisfying the following two restrictions:
    \begin{enumerate}[(R1)]
        \item The fiber of any outmost internal vertex must be of $m-1$ or $m$ leaves.
        \item The fiber of any non-outmost internal vertex contains at most $m$ leaves, {which are all consecutive.}
    \end{enumerate}
Note that in Step 2 of algorithm {\tt STF}, each unlabelled linear tree corresponds to $(n-k+1)!$ labelled linear trees and thus equation \eqref{equ:R-s2F-s} holds.
Further applying the bijection $\psi$, the internal vertices (except the root) and leaves of each labelled linear tree are mapped to asterisked and unasterisked leaves of small trees in the corresponding small forest, respectively. This leads to the two properties (P1) and (P2).

{
Conversely, assume that $F\in \mathcal{F}(n+1,k+1)$ satisfying (P1) and (P2). 
Applying the inverse map of the bijiection $\psi$ to $F$, we see that  the asterisked and unasterisked leaves of small trees in $F$ correspond to the internal vertices (except the root) and leaves of a labelled linear tree, respectively. Let $LT:=\psi^{-1}(F)\in \mathcal{LT}(n-k+1,k+1)$, then it is direct to check that $LT$ satisfies (R1) and (R2) corresponding to the two properties (P1) and (P2) of $F$. 
Denote $T$ the unlabelled linear tree obtained from removing the labels of $LT$. 
The restrictions (R1) and (R2) guarantee the preimage of $T$ with respect to the bijiection $\varphi$ is a saturated $m$-regular simple stack.
}

{
Moreover, note that there is no restriction for the labeling of $LT$. Thus we have $(n-k+1)!$ labelled linear trees with the same configuration as $LT$, which leads to equation \eqref{equ:R-s2F-s} again.
}
\end{proof}

Let $[m,n]^*$ denote the set $\{m^*,(m+1)^*,\ldots,n^*\}$, and $[n]^*=[1,n]^*$. Based on relation \eqref{equ:R-s2F-s}, we can derive the following enumeration result on saturated $m$-regular simple stacks by counting the small forests in $\mathcal{F}_s(n+1,k+1;m)$.
	
\begin{thm} We have
    \begin{align}
	    R_s(n,k;m)=\left[x^{n-2k}y^k\right]\frac{\left(y-y^2+xy+x^{m-1}(1-y)^2-x^{m+1}\right)^{k+1}}{(k+1)(1-x)^{k+1}(1-y)^{2(k+1)}},\label{equ:R-s}
	\end{align}
	where  $n,k\geq 0$.
\end{thm}

\begin{proof}  Let $f(s,t)$ denote the number of fibers constructed by vertices on $[s]\cup [t]^*$ satisfying the properties in Lemma \ref{lem:R-s2F-s}. When $t=0$, by the property (P1) in Lemma \ref{lem:R-s2F-s}, it is easy to see that 
    \begin{equation*}
        f(s,0)=\left\{\begin{aligned}
            \notag &s!,             &&s=m-1\ \text{or}\ s=m,\\
            \notag &0,             &&\text{otherwise}.
        \end{aligned}\right.
    \end{equation*}
    When $t>0$, the property (P2) in Lemma \ref{lem:R-s2F-s} leads to  that 
    \begin{equation*}
        f(s,t)=\left\{\begin{aligned}
            \notag &t!,             &&s=0,\\
            \notag &s!(t+1)!,      &&1\leq s\leq m,\\
            \notag &0,             &&s>m.
        \end{aligned}\right.
    \end{equation*}
    Denote the exponential generating function of $f(s,t)$ by
    \[
    F(x,y;m)=\sum_{s,t\geq 0}f(s,t)\frac{x^s}{s!}\frac{y^t}{t!},
    \] 
    then we have
    \begin{align}
        \notag F(x,y;m)=&f(m-1,0)\frac{x^{m-1}}{(m-1)!}+f(m,0)\frac{x^{m}}{m!}+\sum_{t\geq 1} f(0,t)\frac{y^t}{t!}+\sum_{s=1}^m \sum_{t=1}^\infty f(s,t)\frac{x^s}{s!}\frac{y^t}{t!} \\[5pt]
        \notag=&(m-1)!\frac{x^{m-1}}{(m-1)!}+m!\frac{x^{m}}{m!}+\sum_{t\geq 1}t!\frac{y^t}{t!}+\left(\sum_{s=1}^{m}s!\frac{x^s}{s!}\right)\left(\sum_{t\geq 1}(t+1)!\frac{y^t}{t!}\right)\\[5pt]
    	\notag =&\left(\sum_{s=0}^{m}x^s\right)\left(\sum_{t\geq 0}(t+1)y^t\right)-\sum_{s=0}^{m-2}x^s-\sum_{t\geq 1}ty^t\\[5pt]
    	\notag =&\frac{1-x^{m+1}}{1-x}\cdot\frac{1}{(1-y)^2}-\frac{1-x^{m-1}}{1-x}-\frac{y}{(1-y)^2}\\[5pt]
    	=&\frac{y-y^2+xy+x^{m-1}(1-y)^2-x^{m+1}}{(1-x)(1-y)^2}.
    \end{align}
    Note that a small forest $F\in\mathcal{F}_s(n+1,k+1;m)$ contains $n-k+1$ unasterisked vertices, $n-2k$ unasterisked leaves, and $k$ asterisked leaves. To construct such a small forest, we can first choose $k+1$ unasterisked numbers as the roots' labels, and then choose the $k+1$ fibers, the numbers of which coincide with the generating function $F(x,y;m)$. Therefore
	\begin{align*}
	R_s(n,k;m)=&\frac{{F}_s(n+1,k+1;m)}{(n-k+1)!}\\[5pt]
	=&\frac{1}{(n-k+1)!}\binom{n-k+1}{k+1}\left[\frac{x^{n-2k}}{(n-2k)!}\cdot\frac{y^k}{k!}\right]\left(F(x,y;m)\right)^{k+1}\\[5pt]
	=&\left[x^{n-2k}y^k\right]\frac{\left(y-y^2+xy+x^{m-1}(1-y)^2-x^{m+1}\right)^{k+1}}{(k+1)(1-x)^{k+1}(1-y)^{2(k+1)}}.
	\end{align*}
\end{proof}

Note that an $m$-regular simple stack on $[n]$ contains at most $\lfloor\frac{n-m+1}{2}\rfloor$ arcs, thus $n\geq m-1$ and $0\leq k\leq \lfloor\frac{n-m+1}{2}\rfloor$ is a necessary condition for $R_s(n,k;m)\neq 0$. Specially, for the case of saturated 2-regular and 3-regular simple stacks, \eqref{equ:R-s} reduces to the following explicit formulas.

\begin{cor}\label{cor:R-s(n,k;2)}
	When $n, k\geq 0$, we have
	\begin{align}
	\label{equ:R-s2}    R_s(n,k;2)=&\dfrac{1}{k+1}\sum_{t=1}^{k+1}\binom{k+1}{t}\binom{2t+k-1}{t-1}\binom{t}{n-2k-t},\\
	R_s(n,k;3)=&\sum_{s=2}^{n-2 k} \sum_{t=1}^{s}\frac{(-1)^{s-t-1}}{k+1} \binom{k+1}{n-2 k-s}\binom{k+1}{t}\binom{t}{s-t}\binom{s+k-n-1}{s-t-1}.
	\end{align}
\end{cor}

\begin{proof}
    Setting $m=2$ in \eqref{equ:R-s} leads to
    \begin{align*}
        R_s(n,k;2)=&\left[x^{n-2k}y^k\right]\frac{\left(x^2+x-y^2+y\right)^{k+1}}{(k+1)(1-y)^{2(k+1)}}\\
        =&\left[y^k\right]\frac{\displaystyle\sum_{t=0}^{k+1}\binom{k+1}{t}(y-y^2)^{k+1-t}\left[x^{n-2k}\right](x^2+x)^t}{(k+1)(1-y)^{2(k+1)}}\\
        =&\frac{1}{k+1}\sum_{t=0}^{k+1}\binom{k+1}{t}\left(\left[y^{t-1}\right](1-y)^{-(k+t+1)}\right)\left(\left[x^{n-2k-t}\right](x+1)^t\right)\\
        =&\frac{1}{k+1}\sum_{t=1}^{k+1}\binom{k+1}{t}\binom{2t+k-1}{t-1}\binom{t}{n-2k-t}.
    \end{align*}
    Substituting $m=3$ into \eqref{equ:R-s}, we have
    \begin{align*}
        R_s(n,k;3)=&\left[x^{n-2k}y^k\right]\frac{\left(x^{2}+x y+y\right)^{k+1}(x-y+1)^{k+1}}{(k+1)(1-y)^{2(k+1)}}\\
        =&\left[y^k\right]\frac{\displaystyle\sum_{s=0}^{n-2k}\left(\left[x^{s}\right](x^2+xy+y)^{k+1}\right)\left(\left[x^{n-2k-s}\right](x-y+1)^{k+1}\right)}{(k+1)(1-y)^{2(k+1)}},
    \end{align*}
    where
    \begin{align*}
        \left[x^{s}\right](x^2+xy+y)^{k+1}=&\sum_{t=0}^{k+1}\binom{k+1}{t}\binom{t}{s-t}y^{ t+k-s+1},\\[5pt]
        \left[x^{n-2k-s}\right](x-y+1)^{k+1}=&\binom{k+1}{n-2 k-s}(1-y)^{-(n-3 k-s-1)}.
    \end{align*}
    Hence,
    \begin{align*}
        R_s(n,k;3)=&\sum_{s=0}^{n-2 k} \sum_{t=0}^{k+1}\binom{k+1}{n-2 k-s}\binom{k+1}{t}\binom{t}{s-t}\left(\left[y^k\right]\frac{y^{t+k-s+1}}{(1-y)^{n-k-s+1}}\right)\\[5pt]
        =&\sum_{s=2}^{n-2 k} \sum_{t=1}^{s}\frac{(-1)^{s-t-1}}{k+1} \binom{k+1}{n-2 k-s}\binom{k+1}{t}\binom{t}{s-t}\binom{k+s-n-1}{s-t-1}.
    \end{align*}
\end{proof}

Let $LO_k(n)$ denote the number of $k$-saturated 2-regular simple stacks on $[n]$. According to the definition of $k$-saturated 2-regular simple stacks, we have
\begin{align}
	\label{equ:LOk2Rs}  LO_k(n)&=R_s\left(n,\left\lfloor \frac{n-1}{2}\right\rfloor -k ;2\right).
\end{align}
Hence, when $k=\left\lfloor (n-1)/2\right\rfloor$, \eqref{equ:R-s2} reduces to Clote’s results \cite[Corollary 13]{clote-combinatorics-2006} for the optimal 2-regular simple stacks, see \eqref{equ:LO-0}. Moreover, substituting $k=\left\lfloor (n-1)/2\right\rfloor-1$ and $k=\left\lfloor (n-1)/2\right\rfloor-2$ into \eqref{equ:R-s2}, we have the following results for $1$-saturated and $2$-saturated $2$-regular simple stacks, respectively. 

\begin{cor}
    For any $n\geq 3$, we have
    \begin{align}
        LO_1(n)=&\left\{\begin{aligned}
        	&\frac{(n-1) (n-3) }{192}(n^{2}+8 n+31),&&\text{if}\ n\ \text{is odd},\\[5pt]
        	&\frac{(n-2) (n-4) }{9216}(n^{4}+12 n^{3}+68 n^{2}-288 n-2304), &&\text{if}\ n\ \text{is even},
        \end{aligned}\right.\\
        LO_2(n)=&\left\{\begin{aligned}
        	&\begin{aligned}
        	    &\frac{\left(n-3\right) \left(n-5\right)\left(n-7\right)}{737280} \left(n^{5}+23 n^{4}+278 n^{3}+634 n^{2}\right.\\
        	    &\quad\left.-9879 n-52497\right),
        	\end{aligned}&&\text{if}\ n\ \text{is odd},\\[5pt]
        	&\begin{aligned}
        	    &\frac{\left(n-4\right) \left(n-6\right) \left(n-8\right) }{88473600}\left(n^{7}+28 n^{6}+400 n^{5}-560 n^{4}\right.\\
        	    &\quad\left.-56336 n^{3}-320768 n^{2}+1555200 n+13363200\right),
        	\end{aligned}
        	&&\text{if}\ n\ \text{is even}.
        \end{aligned}\right.
    \end{align}
\end{cor}

Note that for $LO_1(n)$, Clote \cite[Corollary 15]{clote-combinatorics-2006} obtained a recurrence relation which can be reformulated as follows,
\begin{align*}
    &LO_1(2m + 1) = LO_1(2m-1)+\frac{1}{3} m^{3}+\frac{1}{2} m^{2}+\frac{1}{6} m -1,\\
    &LO_1(2m) = LO_1(2m-2) + LO_1(2m-3)+\frac{1}{120} m^{5}+\frac{7}{24} m^{3}-m^{2}\\
    &\qquad\qquad\quad -\frac{3}{10} m+2+\sum_{i=1}^{m-1}\left(LO_1(2i-1) + LO_1(2m-2i-1)\right),
\end{align*}
where $m\geq 3$ and $L O_{1}(0)=L O_{1}(1)=L O_{1}(2)=L O_{1}(3)=L O_{1}(4)=0, LO_{1}(5)=4$.

\section{Enumeration of saturated extended 2-regular simple stacks with given primary component types}\label{sec:stack-given-arc-pattern} 

In this section, we devote to deriving a uniform explicit formula for enumerating saturated extended 2-regular simple stacks with any of the six primary component types. Clote's result \cite{clote-combinatorics-2006} on saturated 2-regular simple stacks is then a consequence of this uniform formula. 

The concept of primary component plays a key role in the enumeration of stacks. Following the structure decomposition idea proposed by Chen et al. \cite{chen-zigzag-2014}, the primary component is defined as the union of the connected components containing vertices 1 and $n$. Following the idea of Guo et al. \cite{guo-number-2022}, the primary component of saturated extended 2-regular simple stacks on $[n]$ can be classified into six types according to the degrees of vertices $1$ and $n$, see Table \ref{tab:extended saturated structures}.

\begin{table}[H]
	\centering
	\caption{Six primary component types of saturated extended 2-regular simple stacks on $[n]$.}\label{tab:extended saturated structures}
	\tikzstyle{nl} = [rectangle, minimum width=11pt,minimum height=8pt, inner sep=0pt,draw]
	\resizebox{\textwidth}{!}{\begin{tabular}{c|c}
			\Xhline{0.3mm}
			($\deg(1), \deg(n)$)&primary component types\\\Xhline{0.3mm}
			\makecell*[c]{(2,0), (0,2)}&
			
			\adjustbox{valign=c}{\begin{tikzpicture}
					\foreach \x in {1,4,7,10}
					\filldraw (\x *0.32,0) circle[radius=1pt];
					
					\draw (1*0.32,-0.2) node{\footnotesize 1}
					(10*0.32,-0.2) node{\footnotesize $n$}
					(5.5*0.32,-0.4) node{\footnotesize $\mathcal{A}_1$}
					
					(1*0.32,0) to [bend left=60] (4*0.32,0)
					(1*0.32,0) to [bend left=60] (7*0.32,0)
					(2.5*0.32,0) node[nl]{\tiny $T_2$}
					(5.5*0.32,0) node[nl]{\tiny $T_1$}
					(8.5*0.32,0) node[nl]{\tiny $T_6$};
			\end{tikzpicture}}\hspace{1cm}
			\adjustbox{valign=c}{\begin{tikzpicture}
					\foreach \x in {1,4,7,10}
					\filldraw (\x *0.32,0) circle[radius=1pt];
					
					\draw (1*0.32,-0.2) node{\footnotesize 1}
					(10*0.32,-0.2) node{\footnotesize $n$}
					(5.5*0.32,-0.4) node{\footnotesize $\mathcal{A}_1'$}
					
					(4*0.32,0) to [bend left=60] (10*0.32,0)
					(7*0.32,0) to [bend left=60] (10*0.32,0)
					(8.5*0.32,0) node[nl]{\tiny $T_2$}
					(5.5*0.32,0) node[nl]{\tiny $T_1$}
					(2.5*0.32,0) node[nl]{\tiny $T_6'$};
			\end{tikzpicture}}\\\hline
			
			\makecell*[c]{(2,1), (1,2) \\ $1,n$ form an arc}&
			
			\adjustbox{valign=c}{\begin{tikzpicture}
					\foreach \x in {1,4,7}
					\filldraw (\x *0.32,0) circle[radius=1pt];
					
					\draw (1*0.32,-0.2) node{\footnotesize 1}
					(7*0.32,-0.2) node{\footnotesize $n$}
					(4*0.32,-0.4) node{\footnotesize $\mathcal{A}_2$}
					
					(1*0.32,0) to [bend left=60] (4*0.32,0)
					(1*0.32,0) to [bend left=60] (7*0.32,0)
					(2.5*0.32,0) node[nl]{\tiny $T_2$}
					(5.5*0.32,0) node[nl]{\tiny $T_6$};
			\end{tikzpicture}}\hspace{1cm}
			\adjustbox{valign=c}{\begin{tikzpicture}
					\foreach \x in {4,7,10}
					\filldraw (\x *0.32,0) circle[radius=1pt];
					
					\draw (4*0.32,-0.2) node{\footnotesize 1}
					(10*0.32,-0.2) node{\footnotesize $n$}
					(7*0.32,-0.4) node{\footnotesize $\mathcal{A}_2'$}
					
					(4*0.32,0) to [bend left=60] (10*0.32,0)
					(7*0.32,0) to [bend left=60] (10*0.32,0)
					(8.5*0.32,0) node[nl]{\tiny $T_2$}
					(5.5*0.32,0) node[nl]{\tiny $T_6'$};
			\end{tikzpicture}}\\\hline
			
			\makecell*[c]{(2,1), (1,2) \\ $1,n$ do not form an arc}&
			
			\adjustbox{valign=c}{\begin{tikzpicture}
					\foreach \x in {1,4,7,10,13}
					\filldraw (\x *0.32,0) circle[radius=1pt];
					
					\draw (1*0.32,-0.2) node{\footnotesize 1}
					(13*0.32,-0.2) node{\footnotesize $n$}
					(7*0.32,-0.4) node{\footnotesize $\mathcal{A}_3$}
					
					(1*0.32,0) to [bend left=60] (4*0.32,0)
					(1*0.32,0) to [bend left=60] (7*0.32,0)
					(10*0.32,0) to [bend left=60] (13*0.32,0)
					(2.5*0.32,0) node[nl]{\tiny $T_2$}
					(5.5*0.32,0) node[nl]{\tiny $T_1$}
					(8.5*0.32,0) node[nl]{\tiny $T_3$}
					(11.5*0.32,0) node[nl]{\tiny $T_7$};
			\end{tikzpicture}}\hspace{1cm}
			\adjustbox{valign=c}{\begin{tikzpicture}
					\foreach \x in {-2,1,4,7,10}
					\filldraw (\x *0.32,0) circle[radius=1pt];
					
					\draw (-2*0.32,-0.2) node{\footnotesize 1}
					(10*0.32,-0.2) node{\footnotesize $n$}
					(4*0.32,-0.4) node{\footnotesize $\mathcal{A}_3'$}
					
					(-2*0.32,0) to [bend left=60] (1*0.32,0)
					(4*0.32,0) to [bend left=60] (10*0.32,0)
					(7*0.32,0) to [bend left=60] (10*0.32,0)
					(-.5*0.32,0) node[nl]{\tiny $T_7'$}
					(2.5*0.32,0) node[nl]{\tiny $T_3$}
					(5.5*0.32,0) node[nl]{\tiny $T_1$}
					(8.5*0.32,0) node[nl]{\tiny $T_2$};
			\end{tikzpicture}}\\\hline
			
			(1,1)&
			
			\adjustbox{valign=c}{\begin{tikzpicture}
					\foreach \x in {1,7}
					\filldraw (\x *0.32,0) circle[radius=1pt];
					
					\draw (1*0.32,-0.2) node{\footnotesize 1}
					(7*0.32,-0.2) node{\footnotesize $n$}
					(4*0.32,-0.4) node{\footnotesize $\mathcal{A}_4$}	
					
					(1*0.32,0) to [bend left=60] (7*0.32,0)
					(4*0.32,0) node[nl]{\tiny $T_5$};	
			\end{tikzpicture}}\\\hline
			
			\makecell*[c]{(2,2) \\ $1,n$ form an arc}&
			
			\adjustbox{valign=c}{\begin{tikzpicture}
					\foreach \x in {1,4,7,10}
					\filldraw (\x *0.32,0) circle[radius=1pt];
					
					\draw (1*0.32,-0.2) node{\footnotesize 1}
					(10*0.32,-0.2) node{\footnotesize $n$}
					(5.5*0.32,-0.4) node{\footnotesize $\mathcal{A}_5$}
					
					(1*0.32,0) to [bend left=60] (4*0.32,0)
					(1*0.32,0) to [bend left=60] (10*0.32,0)
					(7*0.32,0) to [bend left=60] (10*0.32,0)
					(2.5*0.32,0) node[nl]{\tiny $T_2$}
					(5.5*0.32,0) node[nl]{\tiny $T_1$}
					(8.5*0.32,0) node[nl]{\tiny $T_2$};
			\end{tikzpicture}}\\\hline
			
			\makecell*[c]{(2,2) \\ $1,n$ do not form an arc}&
			
			\adjustbox{valign=c}{\begin{tikzpicture}
					\foreach \x in {1,4,7,10,13,16}
					\filldraw (\x *0.32,0) circle[radius=1pt];
					
					\draw (1*0.32,-0.2) node{\footnotesize 1}
					(16*0.32,-0.2) node{\footnotesize $n$}
					(8.5*0.32,-0.4) node{\footnotesize $\mathcal{A}_6$}
					
					(1*0.32,0) to [bend left=60] (4*0.32,0)
					(1*0.32,0) to [bend left=60] (7*0.32,0)
					(10*0.32,0) to [bend left=60] (16*0.32,0)
					(13*0.32,0) to [bend left=60] (16*0.32,0)
					(2.5*0.32,0) node[nl]{\tiny $T_2$}
					(5.5*0.32,0) node[nl]{\tiny $T_1$}
					(8.5*0.32,0) node[nl]{\tiny $T_1$}
					(11.5*0.32,0) node[nl]{\tiny $T_1$}
					(14.5*0.32,0) node[nl]{\tiny $T_2$};
			\end{tikzpicture}}\\\Xhline{0.3mm}
	\end{tabular}}
\end{table}

As shown in Table \ref{tab:extended saturated structures}, the primary component splits $[n]$ into disjoint intervals, each of which contains a substructure. According to the degree and arc length restrictions of saturated extended 2-regular simple stacks, we can classify these substructures into seven types.

{Denote $T$ and $\widehat{T}$ an arbitrary nonempty saturated 2-regular simple stack, and a nonempty saturated 2-regular simple stack with no visible vertex, respectively. Let $\bullet$ and $\varepsilon$ stand for an isolated vertex and an empty stack, respectively.
Then with these notations, the seven types of the substructures in the intervals  can be represented as follows:
\begin{enumerate}[\textbullet]
	\item $T_1 = T + \varepsilon$: $T$ or an empty stack;
	\item $T_2 = T$;
	\item $T_3 = \widehat{T} + \varepsilon$: $\widehat{T}$ or an empty stack;
	\item $T_4 = \widehat{T}$;
	\item $T_5 = \widehat{T} + \bullet$: $\widehat{T}$ or an isolated vertex;
	\item $T_6 = \widehat{T}\bullet + \bullet + \widehat{T} + \varepsilon$: $T_3$ followed by an isolated vertex, or just $T_3$;
	\item $T_7 = \widehat{T} \bullet + \bullet + \widehat{T}$: $T_3$ followed by an isolated vertex, or just $\widehat{T}$.
\end{enumerate}}
Note that the substructures of types $T_1, T_3$ may be empty. $T'_6$ and $T'_7$ stand for the reverse structures of $T_6$ and $T_7$, respectively. 

In the following, we give a semi-bijective algorithm that maps an extended 2-regular simple stack $S$ on $[n]$ with a given primary component type $\mathcal{A}\in\{\mathcal{A}_i\}_{i=1}^6$ to a set of small forests, which is actually a variation of algorithm {\tt STF}. The idea is to preprocess $S$ before applying {\tt STF} and redefine the labeling rules in Step 2 of {\tt STF} to distinguish the primary component from the other parts. Denote this modified algorithm {\tt eSTF} which is stated as follows.

\begin{algorithm}[H]
    \renewcommand{\thealgorithm}{}
	\renewcommand{\algorithmicrequire}{\textbf{Input:}}
	\renewcommand{\algorithmicensure}{\textbf{Output:}}
	\caption{The semi-bijective algorithm {\tt eSTF}}\label{algorithm1}
	\begin{algorithmic}
		\Require  An extended 2-regular simple stack $S$ with $n$ vertices and $k$ arcs, denote $k_1$ the number of arcs in the primary component of $S$.
		\Ensure A set of small forests $\mathcal{F}\subseteq \mathcal{F}(n+d+1,k+1)$.
		\State \textbf{Step 0.} For $v\in \{1,n\}$, if $\deg v=0$, delete $v$ from $S$; if $\deg v=1$, do nothing; if $\deg v=2$, add a new vertex $u$ to the left of $v$ and bond one of the two arcs at $v$ to $u$ so that no crossing occurs. Denote $d=\sum\limits_{v\in \{1,n\}}(\deg v-1)$. Relabel the vertices of $S$ by $[n+d]$ from left to right.
		\State \textbf{Step 1.} Apply the Schmitt-Waterman's bijection on $S$, set $T=\varphi(S)\in\mathcal{T}(n+d-k+1,k+1)$.
		\State \textbf{Step 2.} Denote $P$ the vertex subset of $T$ consisting of the root and the $k_1$ internal vertices corresponding to $k_1$ arcs in the primary component. Label the vertices of $P$ by breadth-first order with $1,\ldots,k_1+1$, and label the other vertices in $T$ by $[k_1+2,n+d-k+1]$ in any of the $(n+d-k-k_1)!$ ways. Denote the set of these $(n+d-k-k_1)!$ labelled linear trees by $\mathcal{LT}$. 
		\State \textbf{Step 3.} Set $\mathcal{F}=\{F \mid F=\psi(T), T \in \mathcal{LT}\}$.
	\end{algorithmic}
\end{algorithm}

The following example illustrates algorithm {\tt eSTF}.
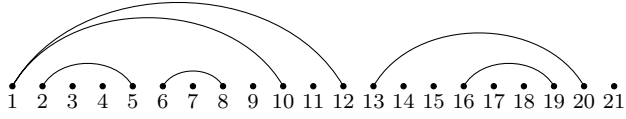
\begin{figure}[H]
	\centering
	\begin{tikzpicture}
		\foreach \x in {1,...,21}
		\filldraw (\x *0.4,0) circle[radius=1pt];
		\foreach \x in {1,...,21}
		\draw (\x*0.4,-0.2) node{\footnotesize \x};
		\draw 
		(1*0.4,0) to [bend left=60] (12*0.4,0)
		(1*0.4,0) to [bend left=60] (10*0.4,0)
		(2*0.4,0) to [bend left=60] (5*0.4,0)
		(6*0.4,0) to [bend left=60] (8*0.4,0)
		(13*0.4,0) to [bend left=60] (20*0.4,0)
		(16*0.4,0) to [bend left=60] (19*0.4,0);
	\end{tikzpicture}
	\caption{A saturated extended 2-regular simple stack.}\label{fig:extended simple stack}
\end{figure}
In the saturated extended 2-regular simple stack in Figure \ref{fig:extended simple stack}, $n=21, k=6, k_1=2, d=0$. By Step 0 and Step 1, we get the simple stack $S$ and its corresponding unlabelled linear tree $\varphi(S)$ as shown in Figure \ref{fig:saturated RNA and linear tree}. The new labeling rule will produce $13!$ different labelled trees from $\varphi(S)$. Figure \ref{fig:meadow-decompose} shows one of the labelled linear trees and its corresponding small forest. 

According to Lemma \ref{lem:R-s2F-s}, saturated 2-regular simple stacks will produce a subset of small forests in which the fiber of any small tree contains no unasterisked vertex, or one unasterisked vertex, or two adjacent unasterisked vertices, and we call such a fiber \textit{saturated}. Corresponding to the substructures of type $T_1,T_2,T_3,T_4$, we can classify the fiber of small trees into the following four types.
\begin{enumerate}[\textbullet]
	\item $F_1$: an arbitrary saturated fiber;
	\item $F_2$: an arbitrary nonempty saturated fiber;
	\item $F_3$: an arbitrary saturated fiber with no unasterisked vertices;
	\item $F_4$: an arbitrary nonempty saturated fiber with no unasterisked vertices;
\end{enumerate}

Let $\mathcal{R}_s(n,k;k_1,d)$ denote the set of extended saturated 2-regular simple stacks on $[n]$ with $k$ arcs whose primary component $\mathcal{A}$ contains $k_1$ arcs and $\sum\limits_{v\in \{1,n\}}(\deg v-1)=d$. Denote the set of small forests corresponding to $\mathcal{R}_s(n,k;k_1,d)$ by
\begin{align*}
   \mathcal{F}_s(n+d+1,k+1;k_1,d)=&\bigcup_{S\in\mathcal{R}_s(n,k;k_1,d)}\text{\tt eSTF}(S).
\end{align*}
Set ${R}_s(n,k;k_1,d)=|\mathcal{R}_s(n,k;k_1,d)|$ and ${F}_s(n+d+1,k+1;k_1,d)=|\mathcal{F}_s(n+d+1,k+1;k_1,d)|$.
Similarly to Lemma \ref{lem:R-s2F-s},  we have the following conclusion.
	
\begin{lem}\label{lem:R-skd2F-s}
   Suppose that $F$ is a small forest. If $F\in\mathcal{F}_s(n+d+1,k+1;k_1,d)$, then 
    \begin{enumerate}[(P1)]
    	\item The fiber of any small tree in $F$ contains no unasterisked vertices, or one unasterisked vertices, or two adjacent unasterisked vertices.
    	\item The vertices $1,2,\ldots,k_1+1$ are roots of the small trees, and the positions of vertices $(n+d+1)^*,\ldots,(n+d-k_1+2)^*$ are fully determined by the primary component $\mathcal{A}$.
    	\item The vertex $(n+d-k_1+1)^*$, if exists, must be a leaf of one of the small trees with roots $1,2,\ldots,k_1+1$.
    \end{enumerate}
    Moreover,
    \begin{equation}\label{equ:Rskd2Fskd}
        R_s(n,k;k_1,d)=\frac{F_s(n+d+1,k+1;k_1,d)}{(n+d-k-k_1)!}.
    \end{equation}
\end{lem}

\begin{proof}
    Assume $S\in \mathcal{R}_s(n,k;k_1,d)$. By Step 0 of {\tt eSTF}, we obtain a saturated simple stack in $\mathcal{R}(n+d,k)$, still denoted by $S$. Step 1 of {\tt eSTF} maps $S$ to a unlabelled tree $T\in\mathcal{T}(n+d-k+1,k+1)$ in which the fiber of any vertex contains no leaf, or one leaf, or two adjacent leaves, which is the case $m=2$ in Lemma \ref{lem:R-s2F-s}, thus property (P1) holds. 
    
    According to Step 2 of {\tt eSTF}, the numbers in $[k_1+1]$ are the labels of {some} internal vertices, and thus they are the roots of {some of} the small trees. Note that the primary component $\mathcal{A}$ corresponds to the vertices in $P$ except the root, which are labelled by $2,\ldots,k_1+1$ in Step 2 and relabelled by $(n+d+1)^*,\ldots,(n+d-k_1+2)^*$ in bijection $\psi$. Thus the positions of vertices $(n+d+1)^*,\ldots,(n+d-k_1+2)^*$ are fully determined by $\mathcal{A}$ and property (P2) follows.
    
    For property (P3), assume that $LT$ is one of the labelled linear trees obtained after Step 2. Note that $(n+d-k_1+1)^*$ is the largest asterisked vertex, except the vertices in $[n+d-k_1+2,n+d+1]^*$, this implies that $(n+d-k_1+1)^*$ is the new label of a vertex whose parent is in $[1,k_1+1]$. Therefore $\psi(LT)$ satisfies the property (P3). 
    
    Equation \eqref{equ:Rskd2Fskd} is straightforward by noting that the vertices of $T$, except those in $P$, are labelled arbitrarily in any of the $(n+d-k-k_1)!$ ways. 
\end{proof}

Lemma \ref{lem:R-skd2F-s} shows that the number of saturated extended 2-regular simple stacks can be obtained by enumerating the corresponding small forests. To this end, the following basic notations and properties on set partitions are prerequisite.

Recall a \textit{partition} of a finite set $S$ is a collection $\pi=\{B_1,B_2,\ldots,B_k\}$ of subsets $B_i\subseteq S$ such that $B_i\neq\varnothing$, $B_i\cap B_j=\varnothing$ for $i\neq j$, and $B_1\cup B_2\cup \ldots \cup B_k=S$. We call $B_i\in \pi$ a \textit{block} of $\pi$. An \textit{ordered partition} is a set partition in which the blocks are linearly ordered.  If the elements of each block of $\pi$ are ordered linearly, we call $\pi$ an \textit{inner-ordered partition} of $S$. An ordered partition is called \textit{dual-ordered} if it is also inner-ordered.
If a partition $\pi$ contains exactly $k$ blocks, we call $\pi$ a $k$-partition.

\begin{lem}\label{lem:composition}
	The number of inner-ordered $k$-partitions of $[n]$ that each block contains at most two elements is $\frac{n!}{k!}\binom{k}{n-k}$.
\end{lem}
\begin{proof}
	To obtain a dual-ordered $k$-partition of $[n]$ that each block contains at most two elements, we first linearly order  $[n]$ in $n!$ ways. Then divide each sequence on $n$ elements into $k$ linearly ordered nonempty blocks such that each block contains at most two elements, which can be obtained through the following generating function
	\[[x^{n}](x+x^2)^k=[x^{n-k}]\sum_{i=0}^{k}\binom{k}{i}x^{i}=\binom{k}{n-k}.\]
	At last, dividing $k!$ to remove the order of the $k$ blocks completes the proof.
\end{proof}

Next, we consider the inner-ordered partitions of a union set consisting of two kinds of elements.  

\begin{lem}\label{lem:order-composition} 
	Let $S=\{s_i\}_{i=1}^s$ and $T=\{t_j\}_{j=1}^k$. Then the number of inner-ordered $k$-partitions on the set $S\cup T$ such that each block contains exactly one element in $T$ is $s!\binom{2k+s-1}{s}.$	
\end{lem}

\begin{proof} Note that any inner-ordered $k$-partition under consideration can be obtained from  a dual-ordered $k$-partition by neglecting the order of the blocks. To construct a dual-ordered $k$-partition, we first construct a sequence of length $2k-1$ consisting of alternatively appeared $t_j\in T$ and vertical bars, then from the $2k$ positions before or after each element of the sequence, choose $s$ positions with repetitions in $\big(\!\binom{2k}{s}\!\big)$ ways to place the elements of $S$.
At last, linearly ordering the elements of $S$ and $T$ in $s!$ and $k!$ ways, respectively, and neglecting the order of the $k$ blocks completes the proof. 
\end{proof}

\begin{lem}\label{lem:p-dual-ordered partition}
	Given $l\geq 1,r\geq 0,u\geq 0,v\geq 0, r+v\geq 1$, denote $C(l,r,u,v)$ the number of dual-ordered $(u+v+r)$-partitions on $[l]\cup[l+1,l+r]^*$ with the following three properties:
	\begin{enumerate}[(a)]
		\item Each block contains no unasterisked element, or one unasterisked element, or two adjacent unasterisked elements.
		\item The first $u$ blocks contain no unasterisked element.
		\item The element $(l+r)^*$ is contained in one of the first $u+v$ blocks.
	\end{enumerate}
	Then we have
	\begin{equation}\label{equ:C(l,r,u,v)}
		C(l,r,u,v)=l!r!\sum_{t=1}^{\min\{l,r+v\}}\binom{t}{l-t}\binom{r+v-1}{t-1}\binom{2t+r}{t-u-v}f(t,r,u,v),
	\end{equation}
	where
	\begin{equation}\label{equ:f(t,r,u,v)}
		f(t,r,u,v)=\frac{ut+(u+v)(t+r+v)}{t(2t+r)}.
	\end{equation}
\end{lem}

\begin{proof}
	To construct such a partition under consideration with $u+v+r$ blocks, assume that $l$ unasterisked vertices are contained by exactly $t$ blocks, and these $t$ blocks contain $s$ asterisked vertices. Then by the properties (a) and (b), it is easy to see that $\lceil \frac{l}{2}\rceil\leq t\leq\min\{l,r+v\}$ and $0\leq s\leq t-u-v$. We denote $p=u+v+r$ for convenience. A dual-ordered partition  $\pi$ can be constructed by the following four steps:
	
	\begin{enumerate}[(1)]
    	\item First construct an inner-ordered $t$-partitions $\pi_1$ of $l$ unasterisked vertices such that each block contains at most two vertices to satisfy the property (a). According to Lemma \ref{lem:composition}, there are $\frac{l!}{t!}\binom{t}{l-t}$ ways.
    	
        \item Insert $s$ unlabelled asterisked vertices into $t$ blocks of $\pi_1$. According to Lemma \ref{lem:order-composition}, there are $\binom{2t+s-1}{s}$ ways to do this. 
        
        \item Divide the remaining $r-s$ unlabelled asterisked vertices into $(p-t)$ blocks to construct a $(p-t)$-partition $\pi_2$. The number of ways to do this is $\frac{1}{(p-t)!}\binom{r-s-1}{p-t-1}$.  
    	
    	\item List the $r$ asterisked vertices. Assume that the first $s$ asterisked vertices are just the ones inserted into the $t$ blocks of $\pi_1$. For each permutation of all $p$ blocks, denote $U$-blocks and $R$-blocks the first $u$ blocks and last $r$ blocks, respectively, let $V$-blocks denote the $v$ blocks between $U$-blocks and $R$-blocks. To meet properties (b) and (c), according to the position of the vertex $(l+r)^*$, we need to discuss the following two cases. 
    	\begin{enumerate}[{Case} 1.]
    	    \item If the vertex $(l+r)^*$ lies in one of the first $s$ positions, then there are $s(r-1)!$ ways to list the asterisked vertices, and the block containing $(l+r)^*$ must contain unasterisked vertices. It is easy to see that $(l+r)^*$ must lie in $V$-blocks of $\pi$. So $\pi$ can be constructed by choosing $u$ blocks from $p-t$ blocks of $\pi_2$ and linearly ordering, then taking the block containing $(l+r)^*$ as one of the $V$-blocks, and permuting the remaining $r+v-1$ blocks.  Thus there are   
    	    \[
    	    v\binom{p-t}{u}u!(r+v-1)!
    	    \]
    	    ways to produce $\pi$.
    	    \item If the vertex $(l+r)^*$ lies in one of the last $r-s$ positions, then there are $(r-s)(r-1)!$ ways to list the asterisked vertices, and the block containing $(l+r)^*$ has no unasterisked vertices. If we take the block containing $(l+r)^*$ as one of the $U$-blocks, there are
    	    \[
    	    u\binom{p-t-1}{u-1}(u-1)!(r+v)!
    	    \] 
    	    ways to produce $\pi$. Otherwise, the block containing $(l+r)^*$ should be one of the $V$-blocks, then the number of ways is \[v\binom{p-t-1}{u}u!(r+v-1)!.\]
    	\end{enumerate}
    \end{enumerate}
	
	Summarizing, we have
	\begin{align}
	\notag &C(l,r,u,v)\\
	\notag =&\sum_{t=\lceil l/2\rceil}^{\textup{min}\{l,r+v\}}\frac{l!(r-1)!}{t!(p-t)!}\binom{t}{l-t}\sum_{s=0}^{t-u-v}\binom{2t+s-1}{s}\binom{r-s-1}{p-t-1}\left(sv\binom{p-t}{u}u!(r+v-1)!\right.\\
	&\quad \left.+(r-s)\left(u\binom{p-t-1}{u-1}(u-1)!(r+v)!+v\binom{p-t-1}{u}u!(r+v-1)!\right)\right).\label{equ:all p}
	\end{align}
	
	For the case of $u>0$, we have $p-t\geq u>0$. Then \eqref{equ:all p} can be simplified as
	\begin{align}\label{equ:p neq t}
	    l!r!\sum_{t=1}^{\textup{min}\{l,r+v\}}\binom{t}{l-t}\binom{r+v-1}{t-1}\sum_{s=0}^{t-u-v}\binom{2t+s-1}{s}\binom{r-s-1}{p-t-1}g(s,t,r,u,v),
	\end{align}
	where
	\[g(s,t,r,u,v)=\frac{sv}{tr}+\frac{(r-s)\left(u(r+v)+v(r+v-t)\right)}{(p-t)tr}.\]
	Applying the following identity \cite[P8 (3b)]{riordan-combinatorial-1968}
	\begin{equation}\label{equ:n+p choose m}
	    \sum_{s=0}^{m}\binom{q+s-1}{s}\binom{n-s}{m-s}=\binom{n+q}{m}, \qquad n, m\geq 0, q\geq 1,
	\end{equation}
	we have 
	\begin{align}
		&\sum_{s=0}^{m}\binom{q+s-1}{s}\binom{n-s}{m-s}s=q\binom{n+q}{m-1},\\
		&\sum_{s=0}^{m}\binom{q+s-1}{s}\binom{n-s}{m-s}(n+1-s)=(n-m+1)\binom{n+q+1}{m}.
	\end{align}
	Setting $q=2t, n=r-1, m=t-u-v$ in the above  two identities and substituting them into the sums involving $s$ in \eqref{equ:p neq t}, it turns to
	\begin{align}
	\notag&l!r!\sum_{t=1}^{\textup{min}\{l,r+v\}}\binom{t}{l-t}\binom{r+v-1}{t-1}\left(\frac{2v}{r}\binom{2t+r-1}{p+t}\right.\\
	\notag&\quad\left.+\frac{u(r+v)+v(r+v-t)}{tr}\binom{2t+r}{p+t}\right)\\
	=&l!r!\sum_{t=1}^{\textup{min}\{l,r+v\}}\binom{t}{l-t}\binom{r+v-1}{t-1}\binom{2t+r}{t-u-v}\frac{ut+(u+v)(t+r+v)}{t(2t+r)}.
	\end{align}
	
	For the case of $u=0$, we divide the summation of \eqref{equ:all p} into two parts as $\lceil l/2\rceil\leq t<r+v=p$ and $t=r+v=p$. As in the case of $u>0$, the part for $\lceil l/2\rceil\leq t<p$ can be simplified to 
	\begin{align}
	    l!r!\sum_{t=1}^{\textup{min}\{l,r+v-1\}}\binom{t}{l-t}\binom{r+v-1}{t-1}\binom{2t+r}{t-v}\frac{v(t+r+v)}{t(2t+r)}.\label{equ:p neq t simple}
	\end{align}
	The part for $t=r+v=p$ equals
	\begin{align}
	    \notag&\frac{l!(r-1)!}{(r+v)!}\binom{r+v}{l-r-v}\binom{3r+2v-1}{r}rv\binom{0}{u}(r+v-1)!\\[5pt]
	    =&l!r!\binom{r+v}{l-r-v}\binom{3r+2v}{r}\frac{2v}{3r+2v},\label{equ:p=t}
	\end{align}
	At last, summing \eqref{equ:p neq t simple} and \eqref{equ:p=t} completes the proof.
\end{proof}

Now we are ready to give the uniform formula for saturated extended 2-regular simple stacks with any of the six primary component types.

\begin{thm}\label{thm:P-V^A}
    Denote $\mathcal{P}(n,k;k_1,d,I_1,I_2,J_1,J_2)$ the set of extended saturated 2-regular simple stacks on $[n]$ with $k$ arcs and {$\mathcal{A}$ the primary component satisfying that}
	\begin{enumerate}[(1)]
	    \item $\mathcal{A}$ contains $k_1$ arcs and $\sum_{v\in \{1,n\}}(\deg v-1)=d$.
	    \item $\mathcal{A}$ splits $[n]$ into disjoint intervals, on which there are $I_1,I_2,J_1,J_2$ substructures of type $T_1,T_2,T_3$ and $T_4$, respectively.
	\end{enumerate}
	Let $P(n,k;k_1,d,I_1,I_2,J_1,J_2)=|\mathcal{P}(n,k;k_1,d,I_1,I_2,J_1,J_2)|$, then
	\begin{equation}\label{equ:P-V^A}
		P(n,k;k_1,d,I_1,I_2,J_1,J_2)=\sum_{i=0}^{I_1}\sum_{j=0}^{J_1}\binom{I_1}{i}\binom{J_1}{j}\frac{C(l,r,j+J_2,i+I_2)}{l!r!},
	\end{equation}
	where $l=n+d-2k$, $r=k-k_1$, and $C(l,r,u,v)$ is defined by \eqref{equ:C(l,r,u,v)}.
\end{thm}

\begin{proof}
	Note that the total number of vertices in the intervals is $l=n+d-2k$.
	For the trivial case $k=k_1$, it is obvious that $J_2=0$, and the $J_1$ intervals of type $T_3$ must be empty. Suppose that there are $i$ nonempty intervals in $I_1$ intervals of type $T_1$, so that $l$ vertices are distributed in $I_2+i$ intervals and these intervals can contain only one or two isolated vertices. Choose $l-I_2-i$ intervals from $I_2+i$ intervals to place two vertices. Therefore
	\begin{equation}\label{equ:extended saturated structures k=k_1}
		P(n,k_1;k_1,d,I_1,I_2,J_1,J_2)=\begin{cases}
			0,&J_2>0,\\
			\displaystyle\sum_{i=0}^{I_1}\binom{I_1}{i}\binom{I_2+i}{l-I_2-i},&J_2=0.
		\end{cases}
	\end{equation}
	
	For the case of $k>k_1$, since $\mathcal{P}(n,k;k_1,d,I_1,I_2,J_1,J_2)\subseteq \mathcal{R}(n,k;k_1,d)$, the small forests corresponding to $\mathcal{P}(n,k;k_1,d,I_1,I_2,J_1,J_2)$ must satisfy conditions (P1)--(P3) in Lemma \ref{lem:R-skd2F-s}. Additionally, from the second restriction for the primary component $\mathcal{A}$, the small forests corresponding to $\mathcal{P}(n,k;k_1,d,I_1,I_2,J_1,J_2)$ should also satisfy the following fourth condition.
	\begin{enumerate}[(P4)]
		\item Ignore the determined vertices $(n+d+1)^*,\ldots,(n+d-k_1+2)^*$. The fiber type of the small trees with roots in $[k_1+1]$ are determined, where there are $I_1,I_2,J_1$, and $J_2$ fibers of type $F_1,F_2,F_3$, and $F_4$, respectively.
	\end{enumerate}  	
	We take four steps to construct the small forests on $[n+d-k+1]$ satisfying those four restrictions.
	\begin{enumerate}[(1)]
		\item Assume there are $i,j$ nonempty fibers in the $I_1,J_1$ fibers of type $F_1$ and $F_3$, respectively. Choose these fibers in $\binom{I_1}{i}\binom{J_1}{j}$ ways.
		\item From the condition (P2) in Lemma \ref{lem:R-skd2F-s}, we have determined $k_1+1$ root labels of small trees corresponding to the primary component $\mathcal{A}$. Select the remaining $k-k_1$ root labels from $[k_1+2,n+d-k+1]$ in $\binom{l+r}{r}$ ways.
		\item Ignore vertices $(n+d+1)^*,\ldots,(n+d-k_1+2)^*$, there are $n+d-2k$ unasterisked vertices and $k-k_1$ asterisked vertices distributed in $r+i+j+I_2+J_2$ saturated fibers. Construct a dual-ordered $(r+i+j+I_2+J_2)$-partition $\pi$ of these vertices with the following  {three restrictions:}
		\begin{enumerate}[(a)]
		    \item Each block contains no unasterisked element, or one unasterisked element, or two adjacent unasterisked elements. 
		    \item The first $j+J_2$ blocks contain no unasterisked element. 
		    \item The element $(n+d-k_1+1)^*$ is contained in one of the first $i+j+I_2+J_2$ blocks.
		\end{enumerate}
		According to Lemma \ref{lem:p-dual-ordered partition}, we have $C(l,r,j+J_2,i+I_2)$ ways to do this.
		\item Denote $\mathcal{R}_I$ and $\mathcal{R}_J$ the sets of the roots of $i+I_2$ fibers of type $F_2$ and $j+J_2$ fibers of type $F_4$, respectively. Let $L_I$ and $L_J$ be the increasing list of the roots in $\mathcal{R}_I$ and $\mathcal{R}_J$, respectively. Allocate $r+i+j+I_2+J_2$ blocks of $\pi$ to the roots in $L_J, L_I$ and $(k_1+2,\ldots, k+1)$ orderly. 
	\end{enumerate}
	It derives that
	\begin{equation}
		(l+r)!P(n,k;k_1,d,I_1,I_2,J_1,J_2)=\sum_{i=0}^{I_1}\sum_{j=0}^{J_1}\binom{I_1}{i}\binom{J_1}{j}\binom{l+r}{r}C(l,r,j+J_2,i+I_2).
	\end{equation}	
	Thus equation \eqref{equ:P-V^A} follows. It is easy to prove that equation \eqref{equ:P-V^A} equals 0 when $k<k_1$ and reduces to equation \eqref{equ:extended saturated structures k=k_1} when $k=k_1$. The theorem therefore holds for all nonnegative integers $k$.
\end{proof}

In fact, the idea of Theorem \ref{thm:P-V^A} is also applicable to enumerate saturated 2-regular simple structures. Note that when $n\geq 1$, a saturated 2-regular simple stack is just a structure of type $T_2$. Taking the primary component to be an empty graph and $k_1,d$ to be 0 in Theorem \ref{thm:P-V^A}, we immediately get
\begin{align*}
	R_s(n,k;2)=&P(n,k;0,0,0,1,0,0),\\
	=&\dfrac{1}{k+1}\sum_{t=1}^{k+1}\binom{k+1}{t}\binom{2t+k-1}{t-1}\binom{t}{n-2k-t},
\end{align*}
which is the same as \eqref{equ:R-s2}.

\section{Enumeration of saturated extended 2-regular simple stacks}\label{sec:saturated-extended-stacks}

In this section, we aim to study the enumeration of the saturated extended 2-regular simple stacks on $[n]$ with $k$ arcs based on Theorem \ref{thm:P-V^A} for such stacks with given primary component types. As consequences, we also obtain enumeration formulas for optimal, 1-saturated, and 2-saturated extended 2-regular simple stacks. 

First, we give the detailed proof of Theorem \ref{thm:ELO(n,k)}, which is a case-by-case application of Theorem \ref{thm:P-V^A} on the six primary component types shown in Table \ref{tab:extended saturated structures}.

\begin{proof}[Proof of \textbf{Theorem \ref{thm:ELO(n,k)}}]
	Denote the number of the saturated extended 2-regular simple stacks on $[n]$ with $k$ arcs and primary component type being $\mathcal{A}_i$ in Table \ref{tab:extended saturated structures} by $s_i(n,k)$. 
	Note that types $\mathcal{A}_i$ and $\mathcal{A}_i'$ ($i=1,2,3$) are symmetric, we will consider $\mathcal{A}_i$ only, therefore,
	\begin{equation}\label{equ:ELO2si}
		ELO(n,k)=2(s_1(n,k)+s_2(n,k)+s_3(n,k))+s_4(n,k)+s_5(n,k)+s_6(n,k).
	\end{equation}
	
	For saturated extended 2-regular simple stacks with primary component of type $\mathcal{A}_1$, the three intervals splitted by the primary component are of type $T_2$, $T_1$ and $T_6$, respectively. If $T_6$ is just $T_3$, it is the case $k_1=2, d=0, I_1=I_2=J_1=1, J_2=0$ in Theorem \ref{thm:P-V^A}. If $T_6$ is $T_3$ followed by an isolated vertex, we delete the isolated vertex and the length of the structure becomes $n-1$. It corresponds to the case $k_1=2, d=0, I_1=I_2=J_1=1, J_2=0$ in Theorem \ref{thm:P-V^A}. Therefore
	\begin{align*}
		&s_1(n,k)\\
		&\quad=P(n, k; 2, 0, 1, 1, 1, 0) + P(n - 1, k; 2, 0, 1, 1, 1, 0)\\
		&\quad=\sum\limits_{t=1}^{k+1}\left(\binom{t}{h}+\binom{t}{h-1}\right)\sum\limits_{i=0}^{1}\sum\limits_{j=0}^{1}\binom{k+i-2}{t-1}\binom{2t+k-2}{t-i-j-1}f(t,k-2,j,1+i),
	\end{align*}
	where $h=n-2k-t$, and $f(t,r,u,v)$ is given by \eqref{equ:f(t,r,u,v)}.
	
	For the case of primary component type being $\mathcal{A}_2$, the two intervals are of type $T_2$ and $T_6$, respectively. Following similar discussions on the interval of type $T_6$, we have
	\begin{align*}		
		s_2(n,k)=&P(n, k; 2, 1, 0, 1, 1, 0) + P(n - 1, k; 2, 1, 0, 1, 1, 0)\\
		=&\frac{1}{k-1}\sum_{t=1}^{k+1}\binom{k-1}{t}\binom{2t+k-1}{t-1}\left(\binom{t}{h+1}+\binom{t}{h}\right).
	\end{align*}
	
	For the case of primary component type being $\mathcal{A}_3$, the four intervals are of type $T_2,T_1,T_3$, and $T_7$, respectively. For the interval of type $T_7$, if it is just $T_4$, it is the case $k_1=3, d=1, I_1=I_2=J_1=J_2=1$ in Theorem \ref{thm:P-V^A}. Otherwise, if it is $T_3$ followed by an isolated vertex. We obtain a structure with $n-1$ vertices by deleting the isolated vertex. It corresponds to the case $k_1=3, d=1, I_1=I_2=1, J_1=2, J_2=0$ in Theorem \ref{thm:P-V^A}. Therefore, 
	\begin{align*}
		s_3(n,k)=&P(n, k; 3, 1, 1, 1, 1, 1) + P(n - 1, k; 3, 1, 1, 1, 2, 0)\\	
		=&\sum_{t=1}^{k+1}\sum_{i=0}^1\sum_{j=0}^{2}\binom{k+i-3}{t-1}\left(\binom{t}{h+1}\binom{1}{j}\binom{2t+k-3}{t-i-j-2}f(t,k-3,1+j,1+i)\right.\\
		&+\left.\binom{t}{h}\binom{2}{j}\binom{2t+k-3}{t-i-j-1}f(t,k-3,j,1+i)\right).
	\end{align*}
	
	Similarly, we also have
	\begin{align*}
		s_4(n,k)=&P(n, k; 1, 0, 0, 0, 0, 1)\\
		=&\sum_{t=1}^{n-2k}\frac{1}{t}\binom{k-2}{t-1}\binom{t}{h}\binom{2t+k-1}{t-1},\\
		s_5(n,k)=&P(n, k; 3, 2, 1, 2, 0, 0)\\
		=&\sum_{t=1}^{k+1}\sum_{i=0}^1\frac{2+i}{t}\binom{k-2+i}{t-1}\binom{t}{h+2}\binom{2t+k-4}{t-2-i},\\
		s_6(n,k)=&P(n, k; 4, 2, 3, 2, 0, 0)\\
		=&\sum_{t=1}^{k+1}\sum_{i=0}^3\frac{2+i}{t}\binom{3}{i}\binom{k-3+i}{t-1}\binom{t}{h+2}\binom{2t+k-5}{t-2-i}.	
	\end{align*}
	Substituting the above expressions for $s_i(n,k), i=1,\ldots,6$ into \eqref{equ:ELO2si}	and simplifying, it leads to \eqref{equ:ELOnk} which completes the proof.
\end{proof}

Let $ELO_k(n)$ denote the number of $k$-saturated extended 2-regular simple stacks on $[n]$. Due to Guo et al. \cite[Lemma 4]{guo-number-2022}, for $n\geq 3$, it holds that
\begin{equation}\label{equ:ELO}
	ELO_k(n)=ELO\left(n,\left\lfloor \frac{n}{2}\right\rfloor -k\right).
\end{equation}

According to \eqref{equ:ELO}, we can obtain the enumeration formula for $k$-saturated extended 2-regular simple stacks from Theorem \ref{thm:ELO(n,k)}. For example, when $k=\left\lfloor \frac{n}{2}\right\rfloor$, \eqref{equ:ELOnk} reduces to the result \eqref{equ:ELO-0} for optimal structures due to Guo et al. \cite{guo-number-2022}. Moreover, for $n\geq 7$, substituting $k=\left\lfloor \frac{n}{2}\right\rfloor -1$ into \eqref{equ:ELOnk}, we obtain the enumeration formula for 1-saturated extended 2-regular simple stacks.
\[ELO_1(n)=\left\{\begin{aligned}
	&\frac{1}{384} n^{5}-\frac{1}{128} n^{4}-\frac{1}{24} n^{3}+\frac{5}{32} n^{2}+\frac{7}{8} n-3,	&	&\textup{if}\ n\ \textup{is even},\\[5pt]
	&\begin{aligned}
		&\frac{1}{23040} n^{7}-\frac{1}{7680} n^{6}-\frac{29}{23040} n^{5}-\frac{23}{1536} n^{4}+\frac{2599}{23040} n^{3}\\[5pt]
		&\quad+\frac{7481}{7680} n^{2}-\frac{46937}{7680} n+\frac{533}{512},
	\end{aligned}	&	&\textup{if}\ n\ \textup{is odd}.
\end{aligned}\right.\]
For $n\geq 8$, substituting $k=\left\lfloor \frac{n}{2}\right\rfloor -2$ into \eqref{equ:ELOnk}, we obtain the enumeration formula for 2-saturated extended 2-regular simple stacks.
\[ELO_2(n)=\left\{\begin{aligned}
	&\begin{aligned}
	    &\frac{1}{2211840} n^{9}-\frac{1}{737280} n^{8}-\frac{1}{46080} n^{7}-\frac{23}{30720} n^{6}+\frac{223}{46080} n^{5}\\[5pt]
	    &\quad+\frac{1397}{15360} n^{4}-\frac{10049}{17280} n^{3}-\frac{6413}{2880} n^{2}+\frac{23}{2} n+27,
	\end{aligned}
	&&\textup{if}\ n\ \textup{is even},\\[5pt]
	&\begin{aligned}
	    &\frac{1}{309657600} n^{11}-\frac{1}{103219200} n^{10}-\frac{1}{4128768} n^{9}
	    -\frac{61}{4128768} n^{8}\\[5pt]
	    &\quad+\frac{5323}{51609600} n^{7}+\frac{21673}{7372800} n^{6}-\frac{558619}{30965760} n^{5}-\frac{143243}{688128} n^{4}\\[5pt]
	    &\quad+\frac{75730687}{103219200} n^{3}+\frac{361742593}{34406400} n^{2}-\frac{32607521}{2293760} n-\frac{14339839}{65536},
	\end{aligned}
	&&\textup{if}\ n\ \textup{is odd}.
\end{aligned}\right.\]

We list some values of $ELO(n,k)$ in Table \ref{tab:ELO}. Note that Theorem \ref{thm:ELO(n,k)} only holds for $k\geq 3$, $ELO(n,1)$ and $ELO(n,2)$ can be obtained by straightforward exhaustive enumeration. 

\begin{table}[H]
	\footnotesize			
	\centering
	\caption{Values of $ELO(n,k)$ for $3\leq n\leq 24$.}\label{tab:ELO}
	\resizebox{\textwidth}{!}{\begin{tabular}{ccccccccccccccc}
			\toprule
			\diagbox[dir=SE]{$k$}{$n$}&	3&	4&	5&	6&	7&	8&	9&	10&	11&	12&	13&	14&	15&	16\\\hline
			1&	1&	&	&	&	&	&	&	&	&	&	&	&	&	\\
			2&	&	2&	7&	9&	8&	6&	2&	&	&	&	&	&	&	\\
			3&	&	&	&	3&	18&	46&	73&	82&	70&	40&	10&	&	&	\\
			4&	&	&	&	&	&	5&	41&	162&	395&	666&	834&	799&	563&	251\\
			5&	&	&	&	&	&	&	&	7&	80&	444&	1534&	3667&	6449&	8690\\
			6&	&	&	&	&	&	&	&	&	&	9&	139&	1026&	4728&	15151\\
			7&	&	&	&	&	&	&	&	&	&	&	&	11&	222&	2099\\
			8&	&	&	&	&	&	&	&	&	&	&	&	&	&	13\\
			sum&	1&	2&	7&	12&	26&	57&	116&	251&	545&	1159&	2517&	5503&	11962&	26204\\			
			\bottomrule				
	\end{tabular}}	
	\resizebox{\textwidth}{2.75cm}{\begin{tabular}{ccccccccc}
			\diagbox[dir=SE]{$k$}{$n$}&	17&	18&	19&	20&	21&	22&	23&	24\\\hline
			4&	50&	&	&	&	&	&	&	\\
			5&	9146&	7403&	4312&	1570&	260&	&	&	\\
			6&	35820&	64919&	92557&	105168&	94660&	65265&	32109&	9875\\
			7&	12362&	50796&	154746&	363026&	673021&	1003604&	1214930&	1191281\\
			8&	333&	3921&	28613&	145817&	553028&	1623141&	3784746&	7141955\\
			9&	&	15&	476&	6827&	60299&	371629&	1708309&	6100976\\
			10&	&	&	&	17&	655&	11239&	117960&	862174\\
			11&	&	&	&	&	&	19&	874&	17676\\
			12&	&	&	&	&	&	&	&	21\\
			sum&	57711&	127054&	280704&	622425&	1381923&	3074897&	6858928&	15323958\\
			\bottomrule				
	\end{tabular}}
\end{table}	

At last, we illustrate the bivariate sequence $ELO(n,k)$ for some $n$ and $k$ in Figure \ref{fig:visualization}. The Maple source codes of this paper can be found at\\
\href{https://github.com/xiaoshuangxiaoshuang/SE2RSS}{https://github.com/xiaoshuangxiaoshuang/SE2RSS}.

\begin{figure}[H]
	\pgfplotsset{
		every axis/.append style={
			font=\small,
			thin,
			tick style={ultra thin}},
	} 
	\centering
	\begin{tikzpicture}
			\begin{axis}[
				xmin=4, xmax=14, 
				title={},
				xlabel={$k$},
				ylabel={$ELO(n,k)$},
				width=.9\textwidth,
				height=.4\textwidth,
				minor tick num=1,
				]
				\addplot [orange,mark=square,mark size=1.3pt] table {
				x	y
				4	0
				5	0
				6	65265
				7	1003604
				8	1623141
				9	371629
				10	11239
				11	19
				12	0
				13	0
				14	0
			};
				\addlegendentry{$n=22$}
				\addplot [green,mark=o,mark size=1.5pt] table {
				x	y
				4	0
				5	0
				6	9875
				7	1191281
				8	7141955
				9	6100976
				10	862174
				11	17676
				12	21
				13	0
				14	0
			};
				\addlegendentry{$n=24$}
				\addplot [magenta,mark=triangle,mark size=1.7pt] table {
				x	y
				4	0
				5	0
				6	0
				7	553448
				8	14098918
				9	40356435
				10	19996555
				11	1853580
				12	26764
				13	23
				14	0
			};
				\addlegendentry{$n=26$}
			\end{axis}
	\end{tikzpicture}
	
	\begin{tikzpicture}
			\begin{axis}[
				xmin=43, xmax=60, 
				title={},
				xlabel={$k$},
				ylabel={$ELO(n,k)$},
				width=.9\textwidth,
				height=.4\textwidth,
				minor tick num=1,
				]
				\addplot [orange,mark=square,mark size=1.3pt] table {
				x	y
				43	65996300000000000000000000000000000000000000000 
				44	1403850000000000000000000000000000000000000000000 
				45	16832900000000000000000000000000000000000000000000 
				46	120691000000000000000000000000000000000000000000000 
				47	541068000000000000000000000000000000000000000000000 
				48	1569880000000000000000000000000000000000000000000000 
				49	3029040000000000000000000000000000000000000000000000 
				50	3971300000000000000000000000000000000000000000000000 
				51	3599380000000000000000000000000000000000000000000000 
				52	2286270000000000000000000000000000000000000000000000 
				53	1028650000000000000000000000000000000000000000000000 
				54	330480000000000000000000000000000000000000000000000 
				55	76245400000000000000000000000000000000000000000000 
				56	12675300000000000000000000000000000000000000000000 
				57	1520260000000000000000000000000000000000000000000 
				58	131423000000000000000000000000000000000000000000 
				59	8162220000000000000000000000000000000000000000 
				60	362125000000000000000000000000000000000000000 
			};
				\addlegendentry{$n=148$}
				\addplot [green,mark=o,mark size=1.5pt] table {
				x	y
				43	52131200000000000000000000000000000000000000000 
				44	1330000000000000000000000000000000000000000000000 
				45	18796500000000000000000000000000000000000000000000 
				46	156840000000000000000000000000000000000000000000000 
				47	810549000000000000000000000000000000000000000000000 
				48	2691770000000000000000000000000000000000000000000000 
				49	5912570000000000000000000000000000000000000000000000 
				50	8789300000000000000000000000000000000000000000000000 
				51	9005990000000000000000000000000000000000000000000000 
				52	6454380000000000000000000000000000000000000000000000 
				53	3272830000000000000000000000000000000000000000000000 
				54	1184600000000000000000000000000000000000000000000000 
				55	308022000000000000000000000000000000000000000000000 
				56	57778600000000000000000000000000000000000000000000 
				57	7834320000000000000000000000000000000000000000000 
				58	767771000000000000000000000000000000000000000000 
				59	54256200000000000000000000000000000000000000000 
				60	2751920000000000000000000000000000000000000000 
			};
				\addlegendentry{$n=149$}
				\addplot [magenta,mark=triangle,mark size=1.7pt] table {
				x	y
				43	38914700000000000000000000000000000000000000000 
				44	1196830000000000000000000000000000000000000000000 
				45	20007400000000000000000000000000000000000000000000 
				46	194770000000000000000000000000000000000000000000000 
				47	1162380000000000000000000000000000000000000000000000 
				48	4423500000000000000000000000000000000000000000000000 
				49	11069600000000000000000000000000000000000000000000000 
				50	18665000000000000000000000000000000000000000000000000 
				51	21623400000000000000000000000000000000000000000000000 
				52	17481700000000000000000000000000000000000000000000000 
				53	9985680000000000000000000000000000000000000000000000 
				54	4068930000000000000000000000000000000000000000000000 
				55	1191240000000000000000000000000000000000000000000000 
				56	251803000000000000000000000000000000000000000000000 
				57	38536500000000000000000000000000000000000000000000 
				58	4272930000000000000000000000000000000000000000000 
				59	342760000000000000000000000000000000000000000000 
				60	19818300000000000000000000000000000000000000000 
			};
				\addlegendentry{$n=150$}
			\end{axis}
	\end{tikzpicture}

    \begin{tikzpicture}
		\begin{axis}[
			xmin=90, xmax=115, 
			title={},
			xlabel={$k$},
			ylabel={$ELO(n,k)$},
			width=.9\textwidth,
			height=.4\textwidth,
			minor tick num=1,
			]
%			\addplot [smooth,orange,mark=triangle,mark size=1.7pt] table {data/ELO_n298.dat};
			\addplot [orange,mark=square,mark size=1.3pt] table {
			x	y
			90	358960000000000000000000000000000000000000000000000000000000000000000000000000000000000000000000000000 
			91	3395300000000000000000000000000000000000000000000000000000000000000000000000000000000000000000000000000 
			92	24801000000000000000000000000000000000000000000000000000000000000000000000000000000000000000000000000000 
			93	141580000000000000000000000000000000000000000000000000000000000000000000000000000000000000000000000000000 
			94	638230000000000000000000000000000000000000000000000000000000000000000000000000000000000000000000000000000 
			95	2293200000000000000000000000000000000000000000000000000000000000000000000000000000000000000000000000000000 
			96	6621500000000000000000000000000000000000000000000000000000000000000000000000000000000000000000000000000000 
			97	15476000000000000000000000000000000000000000000000000000000000000000000000000000000000000000000000000000000 
			98	29471000000000000000000000000000000000000000000000000000000000000000000000000000000000000000000000000000000 
			99	45990000000000000000000000000000000000000000000000000000000000000000000000000000000000000000000000000000000 
			100	59117000000000000000000000000000000000000000000000000000000000000000000000000000000000000000000000000000000 
			101	62884000000000000000000000000000000000000000000000000000000000000000000000000000000000000000000000000000000 
			102	55584000000000000000000000000000000000000000000000000000000000000000000000000000000000000000000000000000000 
			103	40974000000000000000000000000000000000000000000000000000000000000000000000000000000000000000000000000000000 
			104	25272000000000000000000000000000000000000000000000000000000000000000000000000000000000000000000000000000000 
			105	13079000000000000000000000000000000000000000000000000000000000000000000000000000000000000000000000000000000 
			106	5693400000000000000000000000000000000000000000000000000000000000000000000000000000000000000000000000000000 
			107	2089300000000000000000000000000000000000000000000000000000000000000000000000000000000000000000000000000000 
			108	647460000000000000000000000000000000000000000000000000000000000000000000000000000000000000000000000000000 
			109	169700000000000000000000000000000000000000000000000000000000000000000000000000000000000000000000000000000 
			110	37665000000000000000000000000000000000000000000000000000000000000000000000000000000000000000000000000000 
			111	7085300000000000000000000000000000000000000000000000000000000000000000000000000000000000000000000000000 
			112	1130400000000000000000000000000000000000000000000000000000000000000000000000000000000000000000000000000 
			113	153000000000000000000000000000000000000000000000000000000000000000000000000000000000000000000000000000 
			114	17571000000000000000000000000000000000000000000000000000000000000000000000000000000000000000000000000 
			115	1711600000000000000000000000000000000000000000000000000000000000000000000000000000000000000000000000 
		};
			\addlegendentry{$n=298$}
			\addplot [green,mark=o,mark size=1.5pt] table {
			x	y
			90	388470000000000000000000000000000000000000000000000000000000000000000000000000000000000000000000000000 
			91	3983900000000000000000000000000000000000000000000000000000000000000000000000000000000000000000000000000 
			92	31453000000000000000000000000000000000000000000000000000000000000000000000000000000000000000000000000000 
			93	193540000000000000000000000000000000000000000000000000000000000000000000000000000000000000000000000000000 
			94	938270000000000000000000000000000000000000000000000000000000000000000000000000000000000000000000000000000 
			95	3618200000000000000000000000000000000000000000000000000000000000000000000000000000000000000000000000000000 
			96	11192000000000000000000000000000000000000000000000000000000000000000000000000000000000000000000000000000000 
			97	27982000000000000000000000000000000000000000000000000000000000000000000000000000000000000000000000000000000 
			98	56921000000000000000000000000000000000000000000000000000000000000000000000000000000000000000000000000000000 
			99	94775000000000000000000000000000000000000000000000000000000000000000000000000000000000000000000000000000000 
			100	129860000000000000000000000000000000000000000000000000000000000000000000000000000000000000000000000000000000 
			101	147110000000000000000000000000000000000000000000000000000000000000000000000000000000000000000000000000000000 
			102	138380000000000000000000000000000000000000000000000000000000000000000000000000000000000000000000000000000000 
			103	108500000000000000000000000000000000000000000000000000000000000000000000000000000000000000000000000000000000 
			104	71139000000000000000000000000000000000000000000000000000000000000000000000000000000000000000000000000000000 
			105	39124000000000000000000000000000000000000000000000000000000000000000000000000000000000000000000000000000000 
			106	18094000000000000000000000000000000000000000000000000000000000000000000000000000000000000000000000000000000 
			107	7053200000000000000000000000000000000000000000000000000000000000000000000000000000000000000000000000000000 
			108	2321700000000000000000000000000000000000000000000000000000000000000000000000000000000000000000000000000000 
			109	646410000000000000000000000000000000000000000000000000000000000000000000000000000000000000000000000000000 
			110	152420000000000000000000000000000000000000000000000000000000000000000000000000000000000000000000000000000 
			111	30469000000000000000000000000000000000000000000000000000000000000000000000000000000000000000000000000000 
			112	5167400000000000000000000000000000000000000000000000000000000000000000000000000000000000000000000000000 
			113	743830000000000000000000000000000000000000000000000000000000000000000000000000000000000000000000000000 
			114	90895000000000000000000000000000000000000000000000000000000000000000000000000000000000000000000000000 
			115	9427900000000000000000000000000000000000000000000000000000000000000000000000000000000000000000000000 
		};
			\addlegendentry{$n=299$}
			\addplot [magenta,mark=triangle,mark size=1.7pt] table {
			x	y
			90	410050000000000000000000000000000000000000000000000000000000000000000000000000000000000000000000000000 
			91	4562700000000000000000000000000000000000000000000000000000000000000000000000000000000000000000000000000 
			92	38960000000000000000000000000000000000000000000000000000000000000000000000000000000000000000000000000000 
			93	258550000000000000000000000000000000000000000000000000000000000000000000000000000000000000000000000000000 
			94	1348600000000000000000000000000000000000000000000000000000000000000000000000000000000000000000000000000000 
			95	5583500000000000000000000000000000000000000000000000000000000000000000000000000000000000000000000000000000 
			96	18509000000000000000000000000000000000000000000000000000000000000000000000000000000000000000000000000000000 
			97	49512000000000000000000000000000000000000000000000000000000000000000000000000000000000000000000000000000000 
			98	107610000000000000000000000000000000000000000000000000000000000000000000000000000000000000000000000000000000 
			99	191210000000000000000000000000000000000000000000000000000000000000000000000000000000000000000000000000000000 
			100	279280000000000000000000000000000000000000000000000000000000000000000000000000000000000000000000000000000000 
			101	336970000000000000000000000000000000000000000000000000000000000000000000000000000000000000000000000000000000 
			102	337340000000000000000000000000000000000000000000000000000000000000000000000000000000000000000000000000000000 
			103	281300000000000000000000000000000000000000000000000000000000000000000000000000000000000000000000000000000000 
			104	196070000000000000000000000000000000000000000000000000000000000000000000000000000000000000000000000000000000 
			105	114580000000000000000000000000000000000000000000000000000000000000000000000000000000000000000000000000000000 
			106	56294000000000000000000000000000000000000000000000000000000000000000000000000000000000000000000000000000000 
			107	23306000000000000000000000000000000000000000000000000000000000000000000000000000000000000000000000000000000 
			108	8147200000000000000000000000000000000000000000000000000000000000000000000000000000000000000000000000000000 
			109	2408900000000000000000000000000000000000000000000000000000000000000000000000000000000000000000000000000000 
			110	603300000000000000000000000000000000000000000000000000000000000000000000000000000000000000000000000000000 
			111	128120000000000000000000000000000000000000000000000000000000000000000000000000000000000000000000000000000 
			112	23090000000000000000000000000000000000000000000000000000000000000000000000000000000000000000000000000000 
			113	3533400000000000000000000000000000000000000000000000000000000000000000000000000000000000000000000000000 
			114	459250000000000000000000000000000000000000000000000000000000000000000000000000000000000000000000000000 
			115	50697000000000000000000000000000000000000000000000000000000000000000000000000000000000000000000000000 
		};
			\addlegendentry{$n=300$}
		\end{axis}
	\end{tikzpicture}
	\caption{Curves of $ELO(n,k)$.}\label{fig:visualization}
\end{figure}
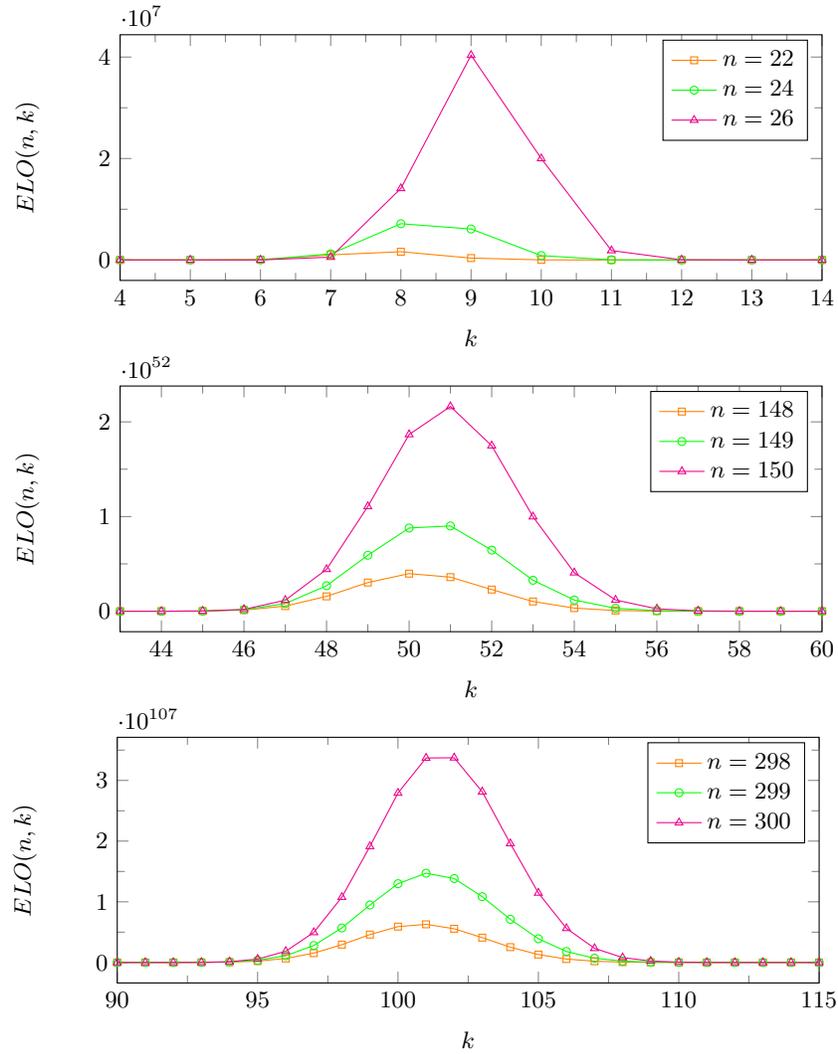

\section*{Acknowledgments}
We are grateful to the anonymous referees for their careful reading and valuable comments and suggestions.
This work was supported by the National Natural Science Foundation of China (Grant No. 12071235 and 11501307), and the Fundamental Research Funds for the Central Universities.

\end{document}